\documentclass[bj,preprint]{imsart}

\usepackage{amsthm,amsmath,amssymb}
\usepackage[round]{natbib}
\usepackage{enumitem}
\RequirePackage[colorlinks,citecolor=blue,urlcolor=blue]{hyperref}

\arxiv{arXiv:1802.08855}

\startlocaldefs
\newcommand{\inv}{^{-1}}

\newcommand{\C}{\mathcal{C}}   % a covering
\newcommand{\D}{\mathcal{D}}   % a packing
\newcommand{\E}{\mathop{\mathbb{E}}}    % Expectation operator
\newcommand{\I}{\mathcal{I}} % Image
\renewcommand{\hat}{\widehat}
\renewcommand{\L}{\mathcal{L}} % Lebesgue space
\newcommand{\M}{\mathcal{M}}    % minimax risk
\newcommand{\N}{\mathbb{N}}    % naturals
\newcommand{\pr}{\mathbb{P}}   % probability
\newcommand{\R}{\mathbb{R}}    % reals
\newcommand{\Res}{\operatorname{Res}}
\renewcommand{\S}{\mathcal{S}}   % First partition
\renewcommand{\SS}{\mathbb{S}}    % Family of all partitions
\newcommand{\T}{\mathcal{T}}   % Second partition
\newcommand{\sminus}{\backslash} % set difference
\renewcommand{\tilde}{\widetilde}
\newcommand{\X}{\mathcal{X}}   % X sample
\newcommand{\Z}{\mathcal{Z}}   % index space
\newcommand{\argmin}{\mathop{\arg\!\min}}
\renewcommand{\epsilon}{\varepsilon}
\newcommand{\IID}{\stackrel{IID}{\sim}}
\renewcommand{\P}{\mathcal{P}} % Family of probability distributions
\newcommand{\Diam}{\operatorname{Diam}} % Diameter of a set
\newcommand{\Sep}{\operatorname{Sep}} % Separation of a set
\newcommand{\Lip}{\operatorname{Lip}} % Class of 1-Lipschitz functions
\endlocaldefs

\newtheorem{theorem}{Theorem}
\newtheorem{lemma}[theorem]{Lemma}

\newenvironment{definition}[1][Definition]{\begin{trivlist}
\item[\hskip \labelsep {\bfseries #1}]}{\end{trivlist}}
\newenvironment{example}[1][Example]{\begin{trivlist}
\item[\hskip \labelsep {\bfseries #1}]}{\end{trivlist}}
\newenvironment{remark}[1][Remark]{\begin{trivlist}
\item[\hskip \labelsep {\bfseries #1}]}{\end{trivlist}}

\newtheorem{innercustomthm}{Theorem}
\newenvironment{customthm}[1]
  {\renewcommand\theinnercustomthm{#1}\innercustomthm}
  {\endinnercustomthm}
\newtheorem{innercustomlemma}{Lemma}
\newenvironment{customlemma}[1]
  {\renewcommand\theinnercustomlemma{#1}\innercustomlemma}
  {\endinnercustomlemma}

\begin{document}

\begin{frontmatter}

% "Title of the Paper"
\title{Minimax Rates of Distribution Estimation\\in Wasserstein Distance}

\runtitle{Minimax Rates of Distribution Estimation in Wasserstein Distance}

\begin{aug}
% indicate corresponding author with \corref{}
% \author{\fnms{John} \snm{Smith}\thanksref{a}\corref{}\ead[label=e1]{smith@foo.com}\ead[label=e2,url]{www.foo.com}}
% \address[a]{\printead{e1};\printead{e2}}

\author{\fnms{Shashank} \snm{Singh}\thanksref{a}\thanksref{b}\corref{}\ead[label=e1]{shashanksi@google.com}}
\and
\author{\fnms{Barnab\'as} \snm{P\'oczos}\thanksref{a}\ead[label=e2]{bapoczos@cs.cmu.edu}}
\address[a]{Machine Learning Department and Department of Statistics \& Data Science, Carnegie Mellon University\\\printead{e1}}
\address[b]{Machine Learning Department, Carnegie Mellon University\\\printead{e2}}

\runauthor{Singh and P\'oczos}

\affiliation{Carnegie Mellon University}

\end{aug}

\begin{abstract}
The Wasserstein metric is an important measure of distance between probability distributions, with many applications in machine learning, statistics, probability theory, and data analysis. This paper provides new upper and lower bounds on statistical minimax rates for the problem of estimating a probability distribution under Wasserstein loss. Specifically, we provide matching rates in a very general setting, using only metric properties, such as covering and packing numbers of balls in the sample space, and moment bounds on the probability distribution.
\end{abstract}


% history:
% \received{\smonth{1} \syear{0000}}

% \tableofcontents

\end{frontmatter}

The Wasserstein metric is an important measure of distance between probability distributions, based on the cost of transforming either distribution into the other through mass transport, under a base metric on the sample space.
Due in part to its intuitive and general nature, the Wasserstein metric has been (re-)discovered many times, and is hence variously attributed to Monge, Kantorovich, Rubinstein, Gini, Mallows, and others; see, for example, Chapter 3 of \citep{villani2008optimalTransport} for a detailed history.
Despite its origins in optimal transport, today it is utilized in such diverse areas as probability theory, statistics, economics, image processing, text mining, robust optimization, and physics~\citep{villani2008optimalTransport,fournier2015rate,esfahani2015robustOptimization,gao2016distributionallyRobust};

In contrast to many other popular notions of dissimilarity
between probability distributions, such as $\L_p$ distances or Kullback-Leibler and other $f$-divergences~\citep{morimoto1963divergences,csiszar1964divergences,ali1966divergences}, which require distributions to be absolutely continuous with respect to each other or to a base measure, Wasserstein distance can be well-defined between \emph{any} pair of probability distributions over a sample space equipped with a metric.
As a particularly important consequence, Wasserstein distances between discrete (e.g., empirical) distributions and continuous distributions are well-defined, finite, and informative (e.g., decay smoothly to $0$ as the distributions become more similar).

Partly for this reason, many central limit theorems and related approximation results~\citep{ruschendorf1985wasserstein,johnson2005central,chatterjee2008normalApproximation,rio2009upper,rio2011asymptotic,chen10SteinsMethod,reitzner2013central} are expressed using Wasserstein distances.
Within machine learning and statistics, this same property motivates a class of so-called \emph{minimum Wasserstein distance estimates}~\citep{del1999CLT,del2003correction,bassetti2006minimum,bernton2017inferenceUsingWasserstein} of distributions, ranging from exponential distributions~\citep{baillo2016exponentialWasserstein} to more exotic models such as restricted Boltzmann machines (RBMs)~\citep{montavon2016wassersteinRBMs} and generative adversarial networks (GANs)~\citep{arjovsky2017wassersteinGAN}.
This class of estimators also includes $k$-means and $k$-medians, where the hypothesis class is taken to be discrete distributions supported on at most $k$ points~\citep{pollard1982quantization}; more flexible algorithms such as hierarchical $k$-means~\citep{ho2017multilevel} and $k$-flats~\citep{tseng2000kFlats} can also be expressed in this way, using a more elaborate hypothesis classes. PCA can also be expressed and generalized (e.g., to manifolds) using Wasserstein distance minimization~\citep{boissard2015template}.
These estimators are conceptually equivalent to empirical risk minimization, leveraging the fact that Wasserstein distances between the empirical distribution and distributions in the relevant hypothesis class are well-behaved.
Moreover, these estimates often perform well in practice because they are free of both tuning parameters and strong distributional assumptions.

For many of the above applications, it is important to understand how quickly the empirical distribution of the data converges to the true population distribution in Wasserstein distance, and whether there exist distribution estimators that converge more quickly. For example, \citet{canas2012learning} and \citet{weed2017sharp} used upper bounds on this rate to prove learning bounds for $k$-means and error bounds for Monte Carlo quadrature, respectively, while \citet{arora2017generalization} used the slow rate of convergence in Wasserstein distance in certain cases to argue that GANs based on Wasserstein distances fail to generalize with fewer than exponentially many samples in the dimension.

To this end, the {\bf main contribution} of this paper is to identify the minimax convergence rate for the problem of estimating a distribution using Wasserstein distance as a loss function. Our setting is very general, relying only on metric properties of the support of the distribution and the number of finite moments the distribution has; some diverse examples to which our results apply are given in Section~\ref{sec:examples}.
Specifically, assuming only that the distribution has some number of finite moments in a given metric, we prove bounds on the minimax convergence rates of distribution estimation, utilizing covering numbers of the sample space for upper bounds and packing numbers for lower bounds.
Thus, we generalize previous upper bounds for this problem, which required that the sample space either be totally bounded or have a linear (Banach space) structure.
Moreover, this paper is the first to study minimax lower bounds for this problem.
Our results show that, without further assumptions on the population distribution, the empirical distribution is typically minimax rate-optimal.

\textbf{Organization: }
The remainder of this paper is organized as follows.
Section~\ref{sec:notation} provides notation required to formally state both the problem of interest and our results, while Section~\ref{sec:related_work} reviews previous work studying convergence of distribution estimates in Wasserstein distance.
Section~\ref{sec:main_results} contains our main upper and lower bound results. Since the proof of the upper bound is fairly long, Appendix~\ref{sec:lemmas} provides a high-level sketch of the proof, followed by a detailed proof in Appendix~\ref{sec:detailed_proof_of_upper_bound} and further lemmas proven in Appendix~\ref{app:proofs}. Proofs of the lower bounds, in terms of packing numbers of the sample space and tails of the distribution, respectively, are given in Appendices \ref{sec:lower_bound_proof_packing_radius} and \ref{sec:lower_bound_proof_moment_bounds}.
Finally, in Section~\ref{sec:examples}, we apply our upper and lower bounds to identify minimax convergence rates in a number of concrete examples.
Section~\ref{sec:conclusion} concludes with a summary of our contributions and suggested avenues for future work.

\section{Notation and Problem Setting}
\label{sec:notation}

In this section, we provide several definitions that are required for formally stating our problem and results.

For any positive integer $n \in \N$, $[n] = \{1,2,...,n\}$ denotes the set of the first $n$ positive integers. For sequences $\{a_n\}_{n \in \N}$ and $\{b_n\}_{n \in \N}$ of non-negative reals, $a_n \lesssim b_n$ and, equivalently $b_n \gtrsim a_n$, indicate the existence of a constant $C > 0$ such that $\limsup_{n \to \infty} \frac{a_n}{b_n} \leq C$. $a_n \asymp b_n$ indicates $a_n \lesssim b_n \lesssim a_n$. Finally, it will be convenient to use the shorthand $X_1^n = X_1,...,X_n$ for a our data sequence.

For the remainder of this paper, fix a metric space $(\Omega,\rho)$, over which $\Sigma$ denotes the Borel $\sigma$-algebra, and let $\P$ denote the family of all Borel probability distributions on $\Omega$.

\subsection{Wasserstein Distance}

The main object of study in this paper is the Wasserstein distance on $\P$, defined as follows:
\begin{definition}[$r$-Wasserstein Distance]
Given two Borel probability distributions $P$ and $Q$ over $\Omega$ and $r \in [1,\infty)$, the $r$-\emph{Wasserstein distance} $W_r(P,Q) \in [0,\infty]$ between $P$ and $Q$ is defined by
\[W_r(P,Q)
  := \inf_{\mu \in \Pi(P,Q)} \left( \E_{(X,Y) \sim \mu} \left[ \rho^r \left( X, Y \right) \right] \right)^{1/r},\]
where $\Pi(P,Q)$ denotes all couplings between $X \sim P$ and $Y \sim Q$; that is,
\[\Pi(P,Q)
  := \left\{ \mu : \Sigma^2 \to [0,1] \middle| \text{ for all } A \in \Sigma,
             \mu(A \times \Omega) = P(A) \text{ and } \mu(\Omega \times A) = Q(A) \right\},\]
is the set of joint probability measures $\mu$ over $\Omega \times \Omega$ with marginals $P$ and $Q$.
\end{definition}
For intuition, in the discrete case, $W_r(P,Q)$ can be thought of as the $r$-weighted total cost of transforming mass distributed according to $P$ to be distributed according to $Q$, where the cost of moving a unit mass from $x \in \Omega$ to $y \in \Omega$ is $\rho(x,y)$. The above definition generalizes this to intuition to arbitrary probability measures. $W_r(P,Q)$ is sometimes defined in terms of equivalent (e.g., dual) formulations; these formulations will not be needed in this paper.
$W_r$ is symmetric in its arguments and satisfies the triangle inequality~\citep{clement2008elementary}, and, for all $P \in \P$, $W_r(P, P) = 0$. Thus, $W_r$ is always a pseudometric. Moreover, it is a proper metric (i.e., $W_r(P,Q) = 0 \Rightarrow P = Q$) if and only if $\rho$ is as well~\citep{villani2008optimalTransport}.

\subsection{Metric Space Definitions}
\label{subsec:definitions}
We now define several notions used to measure the complexity of metric spaces and probability distributions on them.

For any set $\Omega$, $|\Omega|$ denotes the cardinality of $\Omega$, and $2^\Omega$ denotes the power set of $\Omega$.

\begin{definition}[Diameter and Separation of a Set]
For any set $S \subseteq \Omega$, the \emph{diameter $\Diam(S)$ of $S$} $\Diam(S) := \sup_{x,y \in S} \rho(x,y)$ is the largest distance between points in $S$, and the \emph{separation $\Sep(S) := \inf_{x \neq y \in S} \rho(x,y)$ of $S$} is the smallest distance between distinct points in $S$.
\end{definition}

\begin{definition}[Partition of a Set, Resolution of a Partition]
A family $\S \subseteq 2^{\Omega}$ of subsets of $\Omega$ is called a \emph{partition of $\Omega$} if (a) $\Omega = \bigcup_{S \in \S} S$, and (b) all distinct sets $S \neq S' \in \S$ are disjoint (i.e., $S \cap S' = \emptyset$).
If $\S$ is a partition of $\Omega$, then the \emph{resolution $\Res(\S) := \sup_{S \in \S} \Diam(S)$ of $\S$} is the largest diameter of any set in $\S$.
\end{definition}

We now define the covering and packing number of a metric space, which are classic and widely used measures of the size or complexity of a metric space \citep{dudley1967coveringNumbers,haussler1995sphere,zhou2002covering,zhang2002covering}. Our main convergence results will be stated in terms of these quantities, as well as the packing radius, which acts, approximately, as the inverse of the packing number:

\begin{definition}[Covering Number, Packing Number, and Packing Radius of a Metric Space] Fix a set $E \subseteq \Omega$, $\epsilon > 0$, and a positive integer $n$. Then,
\begin{enumerate}
    \item the \emph{$\epsilon$-covering number $N(E,\epsilon) := \min \left\{ |\S| : \S \text{ is a partition of $\E$ with } \Res(\S) \leq \epsilon \right\} \in \N \cup \{\infty\}$ of $E$} is the size of the smallest partition of $E$ with resolution at most $\epsilon$.
    \item the \emph{$\epsilon$-packing number $M(E,\epsilon) := \max \left\{ |S| : S \subseteq E \text{ and } \Sep(S) \geq \epsilon \right\} \in \N \cup \{\infty\}$ of $E$} is the size of the largest subset of $E$ with separation at least $\epsilon$.
    \item the \emph{$n$-packing radius $R(E,n) := \sup\{ \Sep(S) : S \subseteq E \text{ and } |S| \geq n\} \in [0,\infty]$} is the largest possible separation of any $n$-element subset of $E$.
\end{enumerate}
\end{definition}
The covering and packing numbers of a metric space are closely related. Specifically, for any $\epsilon > 0$, we have
\begin{equation}
M(E, \epsilon) \leq N(E, \epsilon) \leq M(E, \epsilon/2).
\label{ineq:covering_packing_relationship}
\end{equation}
The packing number and packing radius also have an approximately inverse relationship: one can check that, for any $\epsilon > 0$ and $n \in \N$,
\begin{equation}
R(E, M(E, \epsilon)) \geq \epsilon
  \quad \text{ and } \quad
  M(E, R(E, n)) \geq n.
\label{ineq:packing_number_radius_relationship}
\end{equation}

\begin{remark}
We defined the covering number slightly differently from usual (using partitions rather than covers). However, the given definition is equivalent to the usual definition, since (a) any partition is itself a cover (i.e., a set $\C \subseteq 2^\Omega$ such that $\Omega \subseteq \bigcup_{C \in \C} C$), and (b), for any countable cover $\C := \{C_1,C_2,...\} \subseteq 2^\Omega$, there exists a partition $\S \in \SS$ with $|\S| \leq |\C|$ and each $S_i \subseteq C_i$, defined recursively by
$S_i := C_i \sminus \bigcup_{j = 1}^{i - 1} S_i$.
$\S$ is often called the \emph{disjointification} of $\C$~\citep{bhattacharya2009stochastic}.
\label{remark:disjointification}
\end{remark}

Finally, since we consider unbounded metric spaces, we will require some sort of concentration conditions on the probability distributions of interest. Specifically, we generalize the notion of the \textit{moments} of a distribution:

\begin{definition}[Metric Moments of a Probability Distribution]
For any $\ell \in [0,\infty]$, probability measure $P \in \P$, and $x \in \Omega$, the \emph{$\ell^{th}$ metric moment $m_{\ell,x}(P)$ of $P$ around $x$} is defined by
\[m_{\ell,x}(P)
  := \left( \E_{Y \sim P} \left[ \left( \rho(x,Y) \right)^\ell \right] \right)^{1/\ell} \in [0,\infty],\]
using the appropriate limit if $\ell = \infty$. For $\mu \geq 0$, we use $\P_{\ell,x,\mu} := \{P \in \P : m_{\ell,x}(P) \leq \mu\}$ to denote the set of Borel probability distributions $P$ on $(\Omega,\rho)$ with $\ell^{th}$ moment around $x$ bounded by $\mu$.

Note that chosen reference point $x$ only affects constant factors since, for all $x,x' \in \Omega$, $\left| m_{\ell,x}^\ell(P) - m_{\ell,x'}^\ell(P) \right| \leq \left( \rho(x,x') \right)^\ell$.
Moreover, if $\Omega$ has linear structure with respect to which $\rho$ is translation-invariant (e.g., if $(\Omega,\rho)$ is a Fr\'echet space~\citep{conway2013course}), we can state our results more simply in terms of $m_\ell(P) := \inf_{x \in \Omega} m_{\ell,x}(P)$.

As an example, in $1$-dimensional Euclidean space $(\Omega,\rho) = (\R,|\cdot - \cdot|)$, $m_2(P)$ is the usual standard deviation of $P$.
\end{definition}

\subsection{Formal Problem Statement}
Fix a metric space $(\Omega,\rho)$ and $r \geq 0$. This paper gives bounds on the minimax risk of estimating a probability distribution from $n$ IID samples, over certain classes $\P$ of distributions on $\Omega$, in $r$-Wasserstein loss $W_r^r$. That is, we upper and lower bound the quantity
\begin{equation}
\M \left( \P, r \right) := \inf_{\hat P} \sup_{P \in \P} \E_{X_1^n \IID P} \left[ W_r^r \left( P, \hat P(X_1^n) \right) \right],
\label{exp:minimax_error}
\end{equation}
where the infimum is taken over all estimators $\hat P$ (i.e., (potentially randomized) functions $\hat P : \Omega^n \to \P$ of the data). In the sequel, for convenience, we suppress the notational dependence of $\hat P = \hat P(X_1^n)$ on $X_1^n$.
Specifically, we will consider the case when $\P = \P_{\ell,x,\mu}$ (for some $\mu \geq 0$, $x \in \Omega$, and $r$ and $\ell$ satisfying $1 \leq r \leq \ell$) contains all Borel distributions with sufficiently bounded moments. In particular, when $\Omega$ is totally bounded, $\P$ contains all Borel distributions on $\Omega$.
Our upper bounds will utilize the empirical distribution
\begin{equation}
    P_n := \frac{1}{n} \sum_{i = 1}^n \delta_{X_i},
    \label{eq:empirical_distribution}
\end{equation}
where $\delta_x$ denotes a Dirac delta mass at $x$, as the estimator $\hat P$. Hence, we will specifically upper bound the quantity
\[\sup_{P \in \P} \E_{X_1^n \IID P} \left[ W_r^r \left( P, P_n \right) \right],\]
which is of interest even outside the minimax context.

\section{Related Work}
\label{sec:related_work}

A long line of work~\citep{dudley1969speed,ajtai1984optimalMatchings,canas2012learning,dereich2013constructive,boissard2014mean,fournier2015rate,weed2017sharp,lei2018convergence} has studied the rate of convergence of the empirical distribution to the population distribution in Wasserstein distance. The most general and tight upper bounds are the recent works of \citet{weed2017sharp}
% , who provide results in general bounded metric spaces,
and \citet{lei2018convergence}.
% , who provide results in unbounded Banach spaces.
As we describe below, while these two papers overlap significantly, neither supersedes the other, and our upper bound combines the key strengths of those in \citet{weed2017sharp} and \citet{lei2018convergence}.

The results of \citet{weed2017sharp} are expressed in terms of a particular notion of dimension, which they call the \textit{Wasserstein dimension} $s$, since they derive convergence rates of order $n^{-r/s}$ (matching the $n^{-r/D}$ rate achieved on the unit cube $[0,1]^D$). The definition of $s$ is complex (e.g., it depends on the sample size $n$), but \citet{weed2017sharp} show that, in many cases, $s$ converges (as $n \to \infty$) to certain common definitions of the intrinsic dimension of the support of the distribution.
The upper bounds in the present paper overcome three main limitations of \citet{weed2017sharp}:
\begin{enumerate}[nolistsep,leftmargin=2em]
\item
The upper bounds of \citet{weed2017sharp} apply only to totally bounded metric spaces. In contrast, our upper bounds apply to unbounded metric spaces under the assumption that the distribution $P$ has some finite moment $m_{\ell,x}(P) < \infty$. The results of \citet{weed2017sharp} correspond to the special case $\ell = \infty$.
\item
Their main upper bound (their Proposition 10) only holds when $s > 2r$, with constant factors diverging to infinity as $s \downarrow 2r$. Hence, their rates are loose when $r$ is large or when the data have low intrinsic dimension.
In contrast, our upper bound is tight even when $s \leq 2r$.
\item
As we discuss in our Example~\ref{ex:Lipschitz_Ball}, the upper bound of \citet{weed2017sharp} can become loose if the Wasserstein dimension $s$ approaches $\infty$ as $n \to \infty$, limiting its utility in infinite-dimensional function spaces. In contrast, we show that our upper and lower bounds match for standard spaces of smooth functions.
\end{enumerate}

On the other hand, \citet{lei2018convergence} focuses on the case where $\Omega$ is a (potentially unbounded and infinite-dimensional) Banach space, under moment assumptions on the distributions. In comparison, our metric space framework has the minor downside that there is no universal reference point (``origin''), which slightly complicates our theorem statements, and that we cannot linearly re-scale the entire space, which slightly complicates our proofs. However, the significant upside is that, for many cases of interest, such as for distributions supported on non-linear manifolds, our rates benefit from properties such as data having low intrinsic dimension.
As a simple example, if the distribution is in fact supported on a finite set of $k$ linearly independent points, the bound of \citet{lei2018convergence} implies only a convergence rate of $n^{-1/k}$, whereas we give a bound of order $O(\sqrt{k/n})$. Our results (unlike those of \citet{lei2018convergence}) also benefit from the multi-scale behavior discussed in Section 5 of \citet{weed2017sharp}, namely, much faster convergence rates are often observed for small $n$ than for large $n$.
These factors may help explain why an algorithm such as functional $k$-means~\citep{garcia2015functionalKMeans} can work in practice, even though the results of \citet{lei2018convergence} suggest only a slow convergence rate of $O\left( (\log n)^{-p} \right)$, for some constant $p > 0$.

Under similarly general (covering number) conditions, \citet{sriperumbudur2010integralProbabilityMetrics,sriperumbudur2012empirical} have studied the related problem of estimating the Wasserstein distance between two unknown distributions given samples from those two distributions.
Since one can estimate Wasserstein distances by plugging in empirical distributions,
our upper bounds imply upper bounds for Wasserstein distance estimation. These bounds are tighter, in several cases, than those of \citet{sriperumbudur2010integralProbabilityMetrics,sriperumbudur2012empirical}; for example, when $\Omega = [0,1]^D$ is the Euclidean unit cube, we give a rate of $n^{-1/D}$, whereas they give a rate of $n^{-\frac{1}{D + 1}}$. Minimax rates for this problem are currently unknown, and it is presently unclear to us under what conditions recent results on estimation of $\L_1$ distances between discrete distributions~\citep{jiao2017minimaxL1} might imply an improved rate as fast as $\left( n \log n \right)^{-1/D}$ for estimation of Wasserstein distance.

Note that, in this paper, we consider only the finite-order Wasserstein distances (i.e., $W_r$ with $r < \infty$). In the Euclidean setting $\Omega \subseteq \R^D$, upper bounds for the case $r = \infty$ have been studied in another long, but essentially disjoint, line of work~\citep{leighton1989tight,shor1991minimax,trillos2015rate,liu2018rate}, with applications ranging from average-case analysis of bin-packing algorithms~\citep{leighton1989tight} to studying consistency of spectral clustering~\citep{trillos2018variational}. Convergence rates for $r = \infty$ typically require assuming that the sample space $\Omega$ is bounded and, in $D$-dimensional Euclidean space, tend to be slower (than in the $r < \infty$) case by a factor of $(\log n)^{1/D}$~\citep{trillos2015rate}.

To the best of our knowledge, minimax lower bounds for distribution estimation under Wasserstein loss remain unstudied, except in the very specific case when $\Omega = [0,1]^D$ is the Euclidean unit cube and $r = 1$~\citep{liang2017well,uppal2019nonparametric}. As noted above, most previous works have focused on studying the convergence rate of the empirical distribution to the true distribution in Wasserstein distance. For this rate, several lower bounds have been established, matching known upper bounds in many cases. However, many distribution estimators besides the empirical distribution can be considered. For example, it is tempting (especially given the infinite dimensionality of the distribution to be estimated) to try to reduce variance by techniques such as smoothing or importance sampling~\citep{bucklew2013introduction}. Our results indicate that the empirical distribution is already minimax optimal, up to constant factors, in many cases.

\section{Main Results}
\label{sec:main_results}

In this section, we present our main upper and lower bounds on the convergence rate of the empirical distribution to the true distribution in Wasserstein distance. Here, only sketches of the proofs of these results are given; however, detailed proofs of the upper bound can be found in Appendices~\ref{sec:lemmas} and \ref{sec:detailed_proof_of_upper_bound}, and proofs of the lower bounds can be found in Appendices~\ref{sec:lower_bound_proof_packing_radius} and \ref{sec:lower_bound_proof_moment_bounds}.

\subsection{Upper Bounds}
We begin with our main upper bound result:

\begin{theorem}[Upper Bound]
Let $x_0 \in \Omega$ and suppose $m_{\ell,x_0}(P) \in [1, \infty)$. Let $J \in \N$ and $\epsilon > 0$.
For each $k \in \N$, define $B_{2^k}(x_0) := \left\{ y \in \Omega : 2^k \leq \rho(x_0, y) < 2^{k + 1} \right\}$. Then, for $\ell \in (r, \infty)\sminus\{2r\}$,
\begin{align}
    \notag
    & \E \left[ W_r^r(P, P_n) \right] \\
    & \leq C_{\ell,r} m_{\ell,x_0}^\ell(P) \left( n^{\frac{r - \ell}{\ell}} + 2^{-2Jr} + \sum_{k \in \N} \sum_{j = 0}^J 2^{(k-2j)r} \min \left\{2^{-k\ell}, \sqrt{\frac{N \left( B_{2^k}(x_0), 2^{k-2j} \right)}{n}} \right\} \right),
    \label{ineq:upper_bound}
\end{align}
where $C_{\ell,r}$ is a constant depending only on $\ell$ and $r$.
Moreover, when $\ell = 2r$, the bound~\eqref{ineq:upper_bound} holds with $n^{\frac{r-\ell}{\ell}}$ replaced by $\frac{\log n}{\sqrt{n}}$.
\label{thm:unbounded_upper_bound}
\end{theorem}

The upper bound~\eqref{ineq:upper_bound} can be thought of as having two main terms: a ``tail'' term of order $n^{\frac{r - \ell}{\ell}}$ and a ``dimensionality'' term, which depends on how the covering numbers $N(B_w(x_0),\eta)$ of balls centered around $x_0$ scale with $w$ and $\eta$, as well as on two free parameters, $J$ and $\epsilon$, which can be chosen (depending on the covering number $N$) to minimize the overall bound. Each of these terms dominates in different settings, and, as discussed below, each matches, up to constant factors, a minimax lower bound on the error of estimating $P$.

The proof of Theorem~\ref{thm:unbounded_upper_bound} involves two main steps, which we sketch here:

\textit{Step 1:}
First, consider the totally bounded case, in which $\Delta := \Diam(\Omega)$ and $N(\Omega,\epsilon)$ are finite for any $\epsilon > 0$. In this setting, one can prove a bound (for any $J \in \N$) of order
\begin{equation}
    \Delta^r 2^{-Jr} + \frac{\Delta^r}{\sqrt{n}} \sum_{j = 1}^J 2^{-2jr} \sqrt{N(\Omega,\Delta 2^{-2j})};
    \label{exp:bounded_bound}
\end{equation}
this is essentially the ``multi-resolution bound'' of \citet{weed2017sharp}, wherein the parameter $J$, controls the number of resolutions considered can be chosen freely to minimize the bound (typically, $J \to \infty$ as $n \to \infty$, at a rate depending on how $N(\Omega,\epsilon)$ scales with $\epsilon$).

\textit{Step 2:}
We now reduce the case of unbounded $\Omega$ to the totally bounded case by partitioning $\Omega$ into a sequence of ``thick spherical shells'' $B_{2^k}(x_0)$, of inner radius $2^k$ and outer radius $2^{k+1}$, centered around $x_0$, and bounding $W_r^r(P,\hat P)$ by a decomposition over these shells. For small $k$, the covering numbers $B_{2^k}(x_0)$ are not too big, and hence we can apply the bound~\eqref{exp:bounded_bound}, leading to the ``dimensionality'' term in~\eqref{ineq:upper_bound}. For large $k$, Markov's inequality and the bounded moment assumption together imply that the probabilities $P(B_{2^k}(x_0)$ and $\hat P(B_{2^k}(x_0)$ decay rapidly; this small amount of mass, which may need to be moved a relatively large distance, leads to the $C_1 n^{\frac{r - \ell}{\ell}}$ ``tail'' term in~\eqref{ineq:upper_bound}. This general strategy of partitioning $\Omega$ into a nested sequence of bounded subsets is similar to that used by \citet{fournier2015rate} and \citet{lei2018convergence}. However, both of these works relied on the assumption that $(\Omega,\rho)$ has a linear (Banach space) structure, which enabled them to use a bound of the form $N(w B, w \epsilon) \leq N(B, \epsilon)$, where $B \subseteq \Omega$ is totally bounded and $w B = \{wx : x \in B\}$ for scalar $w > 0$. This leads to a simpler upper bound, in which the terms depending on $j$ and $k$ can be factored, but, as we discuss in Section~\ref{sec:examples}, requiring $\Omega$ to have linear structure can be limiting.

\subsection{Lower Bounds}

We now turn to providing lower bounds on minimax risk of density estimation in Wasserstein distance; that is, the quantity
\begin{equation}
M(r, \P) = \inf_{\hat P : \Omega^n \to \P} \sup_{P \in \P} \E_{X_1^n \IID P} \left[ W_r^r \left( P, \hat P \right) \right],
\label{exp:distribution_estimation_minimax}
\end{equation}
where the infimum is over all estimators $\hat P$ of $P$ (i.e., all (potentially randomized) functions $\hat P : \Omega^n \to \P$)).

We provide two results: one in terms of packing numbers, for totally bounded metric spaces, and one in terms of the tails of the distribution. Since distributions with totally bounded support necessarily satisfy moment bounds of arbitrary order, in the general unbounded setting with moment constraints, one can apply the maximum of the two bounds.

\begin{theorem}[Minimax Lower Bound in Terms of Packing Radius]
Let $(\Omega,\rho)$ be a metric space, on which $\P$ is the set of Borel probability measures. Then,
\[M(r, \P)
  \geq c_r \sup_{k \in [32n]} R^r(\Omega, k) \sqrt{\frac{k - 1}{n}},\]
where $c_r = \frac{3\log 2}{2^{r + 12}}$ depends only on $r$.
\label{thm:Wasserstein_distribution_estimation_lower_bound}
\end{theorem}

\begin{theorem}[Minimax Lower Bound for Heavy-Tailed Distributions]
Suppose $r, \ell, \mu > 0$ are constants, and fix $x_0 \in \Omega$. Let $\P_{\ell,x_0}(\mu)$ denote the family of distributions $P$ on $\Omega$ with $\ell^{th}$ moment $\mu_{\ell,x_0}(P) \leq \mu$ around $x_0$ at most $\mu$. Let $n \geq \frac{3\mu}{2}$ and assume there exists $x_1 \in \Omega$ such that $\rho(x_0, x_1) = n^{1/\ell}$. Then,
\[M(r, \P_{\ell,x_0}(\mu)) \geq c_\mu n^{\frac{r - \ell}{\ell}},\]
where $c_\mu := \frac{\min \left\{ \mu, 2/3 \right\}}{24}$ is constant in $n$.
\label{thm:heavy_tailed_lower_bound}
\end{theorem}

Recalling that the packing radius $R$ is closely related to the covering number $N$ (via Equations~\eqref{ineq:covering_packing_relationship} and~\eqref{ineq:packing_number_radius_relationship}), one can see that these two bounds correspond to the two ``nonparametric'' terms of the upper bound~\eqref{ineq:upper_bound}. Specifically, it is easy to see that the rate in Theorem~\ref{thm:heavy_tailed_lower_bound} matches the ``tail'' term in~\eqref{ineq:upper_bound}, while it is somewhat less obvious that the simple-looking rate in Theorem~\ref{thm:Wasserstein_distribution_estimation_lower_bound} matches, in many cases of interest, the apparently more complex ``dimension'' term of~\eqref{ineq:upper_bound}. However, as we show in the next section, despite their simplicity, these bounds are indeed tight in many diverse cases of interest.

\section{Examples \& Applications}
\label{sec:examples}

Since our theorems in the previous sections are quite abstract, we conclude by exploring applications of our results to several special cases of interest. In each of the following examples, $P$ is an unknown Borel probability measure over the specified $\Omega$, from which we observe $n$ IID samples $X_1^n \IID P$. The constants $\ell,r,\mu$ are assumed to satisfy $1 \leq r < \ell \leq \infty$ and $0 < \mu < \infty$. $x_0$ can be any point in $\Omega$.

\begin{example}[Finite Space]
Consider the case where $\Omega$ is a finite set, over which $\rho$ is the discrete metric given, for some $\delta > 0$, by
$\rho(x, y) = \delta 1_{\{x = y\}}$, for all $x,y \in \Omega$.
Then, for any $\epsilon \in (0,\delta)$, the covering number is $N(\epsilon) = |\Omega|$. Thus, sending $J \to \infty$ in Theorem~\ref{thm:unbounded_upper_bound}, and setting $k = |\Omega|$ in Theorem~\ref{thm:Wasserstein_distribution_estimation_lower_bound} yields
\[\delta^r \sqrt{\frac{|\Omega| - 1}{n}} \lesssim M(r,\P) \leq \E_{X_1^n \IID P} \left[ W_r^r(P, \hat P) \right] \lesssim \delta^r \sqrt{\frac{|\Omega|}{n}}.\]
\label{ex:discrete_bound}
\end{example}

\begin{example}[Euclidean Space]
Suppose $\Omega = \R^D$ and $\rho$ is the Euclidean metric. Using the fact that $N \left( B_w(0), \epsilon \right) \leq \left( \frac{3 w}{\epsilon} \right)^D$~\citep{pollard1990empirical}, Theorem~\ref{thm:unbounded_upper_bound} gives that, for some constant $C_{D,r,\ell}$ depending only on $D$, $r$, and $\ell$,
\begin{equation}
\E \left[ W_r^r(P, \hat P) \right]
  \leq C_{D,\ell,r} m_\ell^\ell(P) \left(n^{-1/2} + n^{\frac{r - \ell}{\ell}} + 2^{-2Jr} + n^{-1/2} \sum_{j = 1}^J 2^{(D - 2r)j} \right).
  \label{ineq:general_Euclidean_bound}
\end{equation}
Of these three terms, the first depends only on the number $\ell$ of finite moments $P$ and the order $r$ of the Wasserstein distance, whereas the second and third terms depend on choosing the parameter $J$. The optimal choice of $J$ scales with the sample size $n$ at a rate depending on the quantity $D - 2r$. Specifically, if $D = 2r$, then setting $J \asymp \frac{1}{4r} \log_2 n$ gives a rate of
$\E \left[ W_r^r(P, \hat P) \right]
  \lesssim n^{\frac{r - \ell}{\ell}} + n^{-1/2} \log n$.
If $D \neq 2r$, then~\eqref{ineq:general_Euclidean_bound} reduces to
\[\E \left[ W_r^r(P, \hat P) \right]
  \leq C_{D,\ell,r} m_\ell^\ell(P) \left( n^{\frac{r - \ell}{\ell}} + 2^{-2Jr} + n^{-1/2} \frac{2^{(D - 2r)J} - 1}{2^{D - 2r} - 1} \right).\]
Then, if $D < 2r$, sending $J \to \infty$ gives $\E \left[ W_r^r(P, \hat P) \right] \lesssim n^{\frac{r - \ell}{\ell}} + n^{-1/2}$. Finally, if $D > 2r$, then setting $J \asymp \frac{1}{2D} \log n$ gives $\E \left[ W_r^r(P, \hat P) \right] \lesssim n^{\frac{r - \ell}{\ell}} + n^{-\frac{r}{D}}$.
To summarize,
\[\E \left[ W_r^r(P, \hat P) \right]
  \lesssim \left\{
    \begin{array}{ll}
      n^{-1/2} & \text{ if } \ell \in (2r,\infty] \\
      n^{-1/2} \log n & \text { if } \ell = 2r \\
      n^{\frac{r - \ell}{\ell}} & \text { if } \ell \in (r,2r)
    \end{array}
  \right\} + \left\{
    \begin{array}{ll}
      n^{-1/2} & \text { if } r \in (D/2,\infty) \\
      n^{-1/2} \log n & \text { if } r = D/2 \\
      n^{-r/D} & \text { if } r \in [1,D/2)
    \end{array}
  \right.,\]
reproducing Theorem 1 of \citep{fournier2015rate}. Moreover, these rates automatically extend to other spaces $\Omega$ with covering numbers of order $N(B_w(x_0),\epsilon) \in O((w/\epsilon)^D)$, such as any $D$-dimensional manifold with Lipschitz coordinate maps~\citep{eftekhari2017happens}.

On the other hand, it is easy to check that the packing radius $R$ of the unit cube $Q := [0,1]^D$ satisfies $R(Q, n) \geq n^{-1/D}$. Thus, Theorem~\ref{thm:Wasserstein_distribution_estimation_lower_bound} (with, say, $k = n$ and $k = 2$) and Theorem~\ref{thm:heavy_tailed_lower_bound} together yield
\[M(r,\P_{\ell,x_0}(\mu))
  \gtrsim \max \left\{ n^{-r/D}, n^{-1/2}, n^{\frac{r - \ell}{\ell}} \right\}.\]
These upper and lower bound rates match, except in the cases $D = 2r < \ell$ and $\ell = 2r > D/2$, when they differ by a factor of $\log n$.
\citet{ajtai1984optimalMatchings} showed that, for the case $D = 2,r = 1,\ell = \infty$, the empirical distribution converges at the rate $\left( \frac{\log n}{n} \right)^{1/2}$, suggesting that, for $D = 2r < \ell$, our upper and lower bounds may each be loose by a factor of $\sqrt{\log n}$.
\label{ex:unit_cube_lower_bound}
\end{example}

\begin{example}[Unbounded Grid]
This example demonstrates how rates of convergence depend on properties of the metric space $(\Omega,\rho)$ at both large and small scales. Specifically, if we discretize $\Omega$, then the phase transition at $2r = D$ disappears.

Suppose $\Omega = \mathbb{Z}^D$ is a $D$-dimensional grid of integers and $\rho$ is the $\ell_\infty$-metric (given by $\rho(x,y) = \max_{j \in [D]} |x_j - y_j|$). Since $\Z^D \subseteq \R^D$ and the $\ell_\infty$ metric is bounded by the Euclidean metric, the upper bound from the Euclidean case above clearly applies. However, we also have the fact that, whenever $\epsilon < 1$, $N(B_w(0),\epsilon) \asymp w^D$. Therefore, setting $J = 0$ and sending $\epsilon \to 0$ in Theorem~\ref{thm:unbounded_upper_bound} gives, for a constant $C_{D,\ell,r}$ depending only on $D$, $\ell$, and $r$,
\begin{align*}
\E \left[ W_r^r(P, \hat P) \right]
& \leq C_{D,\ell,r} m_\ell^\ell(P) \left( n^{\frac{r - \ell}{\ell}} + n^{-1/2} \right).
\end{align*}
When $\ell > \frac{D + 1}{D} r$ and $r < 2D$, this rate is faster than the general rate shown above for Euclidean spaces.
To the best of our knowledge, no prior results in the literature imply this fact.
\end{example}

\begin{example}[Latent Variable Models, Manifolds]
This example demonstrates that the convergence rate of the empirical distribution in Wasserstein distance improves in the presence of additional structure in the data. Importantly, no \emph{knowledge} of this structure is needed to obtain this accelerated convergence, since it is inherent to the empirical distribution itself.

Suppose that there exist a metric space $(\Z, \rho_\Z)$, a $L$-Lipschitz mapping $\phi : \mathcal{Y} \to \Omega$, and a probability distribution $Q$ on $\Z$ such that $P$ is the pushforward on $Q$ under $\phi$; i.e., for any $A \subseteq \Omega$, $P(A) = Q(f\inv(A))$, where $\phi\inv(A)$ denotes the pre-image of $A$ under $\phi$. This setting is inherent, for example, in many latent variable models. When $\Z \subseteq \R^d$ and $\Omega \subseteq \R^D$ with $d < D$, this generalizes the assumption, popular in high-dimensional nonparametric statistics, that the data lie on a low-dimensional manifold.

In this setting, one can easily bound moments of $P$ and covering numbers in $\Omega$ in terms of those of $Q$ and in $\Z$, respectively. Specifically,
\begin{enumerate}[nosep]
    \item[(a)] for any $z \in \Z$, $\ell > 0$, $m_{\ell,\phi(z)}(P) \leq L m_{\ell,z}(Q)$, and
    \item[(b)] for any $E \subseteq \Omega$, $\epsilon > 0$, $N(E,\rho,\epsilon) \leq N(\phi\inv(E),\rho_\Z,\epsilon/L)$.
\end{enumerate}
This allows us to bound convergence rates over $\Omega$ in terms of moment bounds on $Q$ and covering number bounds on $(\Z,\rho_\Z)$. For example, if $\Z \subseteq \R^d$ and $\rho_\Z$ is the Euclidean metric, then, for any bounded $E \subseteq \Z$, we necessarily have $N(E,\rho_\Z,\epsilon) \in O \left( \epsilon^{-d} \right)$ as $\epsilon \to 0$. If $\Omega \subseteq \R^D$ with $d < D$, then, via analysis similar to that in the Euclidean case above, Theorem~\ref{thm:unbounded_upper_bound} gives a convergence rate of $n^{-1/2} n^{\frac{r - \ell}{\ell}} + n^{-r/d}$, potentially much faster than the $n^{-1/2} n^{\frac{r - \ell}{\ell}} + n^{-r/D}$ minimax lower bound that can be derived without assuming this low-dimensional structure.
\end{example}

\begin{example}[H\"older Ball, $\L_\infty$ Metric]
Finally, we consider distributions over an infinite dimensional space of smooth functions.

Suppose that, for some $\alpha \in (0,1]$,
\[\Omega
  := \left\{ f : [0,1]^D \to [-1,1] \quad \middle| \quad \forall x,y \in [0,1]^D, \quad |f(x) - f(y)| \leq \|x - y\|_2^\alpha \right\}\]
is the class of unit $\alpha$-H\"older functions on the unit cube and $\rho$ is the $\L^\infty$-metric given by
\[\rho(f,g) = \sup_{x \in [0,1]^D} |f(x) - g(x)|, \quad \text{ for all } f,g \in \Omega.\]
The covering and packing numbers of $(\Omega,\rho)$ are known to be of order $\exp \left( \epsilon^{-D/\alpha} \right)$ \citep{devore1993approximation}; specifically, there exist positive constants $0 < c_1 < c_2$ such that, for all $\epsilon \in (0,1)$,
\[c_1 \exp \left( \epsilon^{-D/\alpha} \right)
  \leq M(\epsilon)
  \leq N(\epsilon)
  \leq c_2 \exp \left( \epsilon^{-D/\alpha} \right).\]
Since $\Diam(\Omega) < \infty$, applying Theorem~\ref{thm:unbounded_upper_bound} with $J = 0$ and
$\epsilon = \left( \log n \right)^{-\alpha/D}$ and Theorem~\ref{thm:Wasserstein_distribution_estimation_lower_bound} with $k \asymp n$ yields
\[\left( \log n \right)^\frac{-\alpha r}{D}
  \lesssim M(r,\P)
  \leq \E \left[ W_r^r(P, \hat P) \right]
  \lesssim \left( \log n \right)^{\frac{- \alpha r}{D}}.\]
Conversely, Inequality~\eqref{ineq:packing_number_radius_relationship} implies $R(n) \geq \left( \log(n/c_1) \right)^\frac{-\alpha}{D}$, and so setting $k = n$ in Theorem~\ref{thm:Wasserstein_distribution_estimation_lower_bound} gives
that distribution estimation over $(\P,W_r^r)$ has the extremely slow minimax rate $\left( \log n \right)^\frac{-\alpha r}{D}$. Although we considered only $\alpha \in (0,1]$ (due to the notational complexity of defining higher-order H\"older spaces), analogous rates hold for all $\alpha > 0$. Also, since our rates depend only on covering and packing numbers of $\Omega$, identical rates can be derived for related Sobolev and Besov classes.
Note that the Wasserstein dimension used in the prior work \citep{weed2017sharp} is of order $\frac{D}{\alpha} \log n$, and so their upper bound (their Proposition 10) gives a rate of $n^{-\frac{\alpha r}{D \log n}} = \exp \left( -\frac{\alpha r}{D} \right)$, which fails to converge as $n \to \infty$.
\label{ex:Lipschitz_Ball}
\end{example}

One might wonder why we are interested in studying Wasserstein convergence of distributions over spaces of smooth functions, as in Example~\ref{ex:Lipschitz_Ball}.
Motivation comes from the historical use of smooth function spaces as models for images and other complex naturalistic signals \citep{mallat1999wavelet,peyre2011numerical,sadhanala2016totalVariation}.
Empirical breakthroughs have recently been made in generative modeling, particularly of images, based on the principle of minimizing Wasserstein distance between the empirical distribution and a large class of models encoded by a deep neural network~\citep{montavon2016wassersteinRBMs,arjovsky2017wassersteinGAN,gulrajani2017improved}.

However, little is known about theoretical properties of these methods; while there has been some work studying the optimization landscape of such models~\citep{nagarajan2017gradient,liang2018interaction}, we know of far less work exploring their \textit{statistical} properties.
Given the extremely slow minimax convergence rate we derived above, it must be the case that the class of distributions encoded by such models is far smaller than $\P$. An important avenue for further work is thus to explicitly identify stronger assumptions that can be made on distributions over interesting classes of signals, such as images, to bridge the gap between empirical performance and our theoretical understanding.

\begin{example}[Expectations of Lipschitz Functions \& Monte Carlo Integration]
A fundamental statistical problem is to estimate an expectation $\E_{X \sim P} \left[ f(X) \right]$ of some function $f$ with respect to a distribution $P$. A classical duality result of Kantorovich~\citep{kantorovich1942translocation} implies that
\[W_1(P,Q) = \sup_{f \in \Lip(\Omega)} \left| \E_{X \sim P} \left[ f(X) \right] - \E_{Y \sim Q} \left[ f(Y) \right] \right|,\]
where
\[\Lip(\Omega) := \left\{ f : \Omega \to \R : \sup_{x \neq y \in \Omega} \frac{|f(x) - f(y)|}{\rho(x,y)} \leq 1 \right\}\]
denotes the class of $1$-Lipschitz functions on $(\Omega,\rho)$.
Our upper bound (Theorem~\ref{thm:unbounded_upper_bound}) thus implies bounds, uniformly over $1$-Lipschitz functions $f$, on the expected error of estimating an expectation $\E_{X \sim P} \left[ f(X) \right]$ by the empirical estimate $\frac{1}{n} \sum_{i = 1}^n f(X_i)$ based on $X_1^n \IID P$.
Moreover, our lower bounds (Theorems~\ref{thm:Wasserstein_distribution_estimation_lower_bound} and~\ref{thm:heavy_tailed_lower_bound}) imply that this empirical estimate is minimax rate-optimal over $P$ satisfying only bounded moment assumptions.

As \citet{weed2017sharp} noted, this has consequences for Monte Carlo integration, a common approach to numerical integration in which an integral $\int_\Omega f \, d\lambda$ of a function $f$ with respect to a measure $\lambda$ is estimated based on $n$ IID samples from a probability distribution $P$ proportional to $\lambda$; Monte Carlo integration is useful even when $f$ and $\lambda$ are known analytically, since numerically computing this integral can be challenging, especially in high dimensions or when the supports of $f$ and $\lambda$ are unbounded. In this context, the sample size $n$ required to obtain a desired accuracy directly determines the computational demand of the integration scheme.

Our upper bounds allow one to generalize the upper bound of \citet{weed2017sharp} for Monte Carlo integration (their Proposition 21) to the important case of integrals over unbounded domains $\Omega$, and, moreover, our lower bounds imply that, at least without further knowledge of $f\in \Lip(\Omega)$ and $P \in \P_{\ell,x_0}(\mu)$, the empirical estimate above is rate-optimal among Monte Carlo estimates (i.e., among functions of $X_1^n$). Although improved estimates can be constructed for specific $f$, $\Omega$, and $\lambda$, these worst-case results are useful when either $f$ or $\lambda$ is too complex to model analytically, as often happens, for example, in Bayesian inference problems~\citep{geweke1989bayesian}.
\end{example}

\section{Conclusion}
\label{sec:conclusion}

In this paper, we derived upper and lower bounds for estimating a probability distribution under Wasserstein loss. Our upper bounds generalize and tighten several prior results on the convergence the empirical distribution, while our lower bounds are essentially the first minimax lower bounds for this problem. We also provided several concrete examples in which our bounds imply novel convergence rates.

We studied minimax rates over the very large entire class $\P$ of all distributions with some number of finite moments.
It would be useful to understand how minimax rates improve when additional assumptions, such as smoothness, are made (see, e.g., \citep{liang2017well,singh2018adversarial,uppal2019nonparametric} for somewhat improved upper bounds under smoothness assumptions when $(\Omega,\rho)$ is the Euclidean unit cube and $r = 1$).
Given the slow convergence rates we found over $\P$ in many cases, studying minimax rates under stronger assumptions may help to explain the relatively favorable empirical performance of popular distribution estimators based on empirical risk minimization in Wasserstein loss.
Moreover, while rates over all of $\P$ are of interest only for very weak metrics such as the Wasserstein distance (as stronger metrics may be infinite or undefined), studying minimax rates under additional assumptions will allow for a better understanding of the Wasserstein metric in relation to other commonly used metrics.

\section*{Acknowledgements}

This research was supported by grants from the National Science Foundation (award numbers DGE1252522 and DGE1745016) and the Richard King Mellon Foundation.

\bibliographystyle{plainnat}
\bibliography{sample}

\begin{thebibliography}{69}
\providecommand{\natexlab}[1]{#1}
\providecommand{\url}[1]{\texttt{#1}}
\expandafter\ifx\csname urlstyle\endcsname\relax
  \providecommand{\doi}[1]{doi: #1}\else
  \providecommand{\doi}{doi: \begingroup \urlstyle{rm}\Url}\fi

\bibitem[Ajtai et~al.(1984)Ajtai, Koml{\'o}s, and
  Tusn{\'a}dy]{ajtai1984optimalMatchings}
Mikl{\'o}s Ajtai, J{\'a}nos Koml{\'o}s, and G{\'a}bor Tusn{\'a}dy.
\newblock On optimal matchings.
\newblock \emph{Combinatorica}, 4\penalty0 (4):\penalty0 259--264, 1984.

\bibitem[Ali and Silvey(1966)]{ali1966divergences}
Syed~Mumtaz Ali and Samuel~D Silvey.
\newblock A general class of coefficients of divergence of one distribution
  from another.
\newblock \emph{Journal of the Royal Statistical Society. Series B
  (Methodological)}, pages 131--142, 1966.

\bibitem[Arjovsky et~al.(2017)Arjovsky, Chintala, and
  Bottou]{arjovsky2017wassersteinGAN}
Martin Arjovsky, Soumith Chintala, and L{\'e}on Bottou.
\newblock Wasserstein generative adversarial networks.
\newblock In \emph{International Conference on Machine Learning}, pages
  214--223, 2017.

\bibitem[Arora et~al.(2017)Arora, Ge, Liang, Ma, and
  Zhang]{arora2017generalization}
Sanjeev Arora, Rong Ge, Yingyu Liang, Tengyu Ma, and Yi~Zhang.
\newblock Generalization and equilibrium in generative adversarial nets (gans).
\newblock In \emph{International Conference on Machine Learning}, pages
  224--232, 2017.

\bibitem[Ba{\'\i}llo et~al.(2016)Ba{\'\i}llo, C{\'a}rcamo, and
  Getman]{baillo2016exponentialWasserstein}
Amparo Ba{\'\i}llo, Javier C{\'a}rcamo, and Konstantin~V Getman.
\newblock The estimation of {W}asserstein and {Z}olotarev distances to the
  class of exponential variables.
\newblock \emph{arXiv preprint arXiv:1603.06806}, 2016.

\bibitem[Bassetti et~al.(2006)Bassetti, Bodini, and
  Regazzini]{bassetti2006minimum}
Federico Bassetti, Antonella Bodini, and Eugenio Regazzini.
\newblock On minimum {K}antorovich distance estimators.
\newblock \emph{Statistics \& probability letters}, 76\penalty0 (12):\penalty0
  1298--1302, 2006.

\bibitem[Berend and Kontorovich(2013)]{berend2013binomialMAD}
Daniel Berend and Aryeh Kontorovich.
\newblock A sharp estimate of the binomial mean absolute deviation with
  applications.
\newblock \emph{Statistics \& Probability Letters}, 83\penalty0 (4):\penalty0
  1254--1259, 2013.

\bibitem[Bernton et~al.(2017)Bernton, Jacob, Gerber, and
  Robert]{bernton2017inferenceUsingWasserstein}
Espen Bernton, Pierre~E Jacob, Mathieu Gerber, and Christian~P Robert.
\newblock Inference in generative models using the {W}asserstein distance.
\newblock \emph{arXiv preprint arXiv:1701.05146}, 2017.

\bibitem[Bhattacharya and Waymire(2009)]{bhattacharya2009stochastic}
Rabi~N Bhattacharya and Edward~C Waymire.
\newblock \emph{Stochastic processes with applications}, volume~61.
\newblock Siam, 2009.

\bibitem[Boissard et~al.(2014)Boissard, Le~Gouic, et~al.]{boissard2014mean}
Emmanuel Boissard, Thibaut Le~Gouic, et~al.
\newblock On the mean speed of convergence of empirical and occupation measures
  in {W}asserstein distance.
\newblock \emph{Annales de l'Institut Henri Poincar{\'e}, Probabilit{\'e}s et
  Statistiques}, 50\penalty0 (2):\penalty0 539--563, 2014.

\bibitem[Boissard et~al.(2015)Boissard, Le~Gouic, Loubes,
  et~al.]{boissard2015template}
Emmanuel Boissard, Thibaut Le~Gouic, Jean-Michel Loubes, et~al.
\newblock Distribution’s template estimate with {W}asserstein metrics.
\newblock \emph{Bernoulli}, 21\penalty0 (2):\penalty0 740--759, 2015.

\bibitem[Bucklew(2013)]{bucklew2013introduction}
James Bucklew.
\newblock \emph{Introduction to rare event simulation}.
\newblock Springer Science \& Business Media, 2013.

\bibitem[Canas and Rosasco(2012)]{canas2012learning}
Guillermo Canas and Lorenzo Rosasco.
\newblock Learning probability measures with respect to optimal transport
  metrics.
\newblock In \emph{Advances in Neural Information Processing Systems}, pages
  2492--2500, 2012.

\bibitem[Chatterjee et~al.(2008)]{chatterjee2008normalApproximation}
Sourav Chatterjee et~al.
\newblock A new method of normal approximation.
\newblock \emph{The Annals of Probability}, 36\penalty0 (4):\penalty0
  1584--1610, 2008.

\bibitem[Chen et~al.(2010)Chen, Goldstein, and Shao]{chen10SteinsMethod}
Louis~HY Chen, Larry Goldstein, and Qi-Man Shao.
\newblock \emph{Normal approximation by {S}tein’s method}.
\newblock Springer Science \& Business Media, 2010.

\bibitem[Clement and Desch(2008)]{clement2008elementary}
Philippe Clement and Wolfgang Desch.
\newblock An elementary proof of the triangle inequality for the {W}asserstein
  metric.
\newblock \emph{Proceedings of the American Mathematical Society}, 136\penalty0
  (1):\penalty0 333--339, 2008.

\bibitem[Conway(2013)]{conway2013course}
John~B Conway.
\newblock \emph{A course in functional analysis}, volume~96.
\newblock Springer Science \& Business Media, 2013.

\bibitem[Csisz{\'a}r(1964)]{csiszar1964divergences}
Imre Csisz{\'a}r.
\newblock Eine informationstheoretische ungleichung und ihre anwendung auf
  beweis der ergodizitaet von markoffschen ketten.
\newblock \emph{Magyer Tud. Akad. Mat. Kutato Int. Koezl.}, 8:\penalty0
  85--108, 1964.

\bibitem[del Barrio et~al.(1999)del Barrio, Gin{\'e}, and
  Matr{\'a}n]{del1999CLT}
Eustasio del Barrio, Evarist Gin{\'e}, and Carlos Matr{\'a}n.
\newblock Central limit theorems for the {W}asserstein distance between the
  empirical and the true distributions.
\newblock \emph{Annals of Probability}, pages 1009--1071, 1999.

\bibitem[del Barrio et~al.(2003)del Barrio, Gin{\'e}, Matr{\'a}n,
  et~al.]{del2003correction}
Eustasio del Barrio, Evarist Gin{\'e}, Carlos Matr{\'a}n, et~al.
\newblock Correction: Central limit theorems for the {W}asserstein distance
  between the empirical and the true distributions.
\newblock \emph{The Annals of Probability}, 31\penalty0 (2):\penalty0
  1142--1143, 2003.

\bibitem[Dereich et~al.(2013)Dereich, Scheutzow, Schottstedt,
  et~al.]{dereich2013constructive}
Steffen Dereich, Michael Scheutzow, Reik Schottstedt, et~al.
\newblock Constructive quantization: Approximation by empirical measures.
\newblock \emph{Annales de l'Institut Henri Poincar{\'e}, Probabilit{\'e}s et
  Statistiques}, 49\penalty0 (4):\penalty0 1183--1203, 2013.

\bibitem[DeVore and Lorentz(1993)]{devore1993approximation}
Ronald~A DeVore and George~G Lorentz.
\newblock \emph{Constructive approximation}, volume 303.
\newblock Springer Science \& Business Media, 1993.

\bibitem[Do~Ba et~al.(2011)Do~Ba, Nguyen, Nguyen, and
  Rubinfeld]{do2011sublinearTimeEMD}
Khanh Do~Ba, Huy~L Nguyen, Huy~N Nguyen, and Ronitt Rubinfeld.
\newblock Sublinear time algorithms for earth mover’s distance.
\newblock \emph{Theory of Computing Systems}, 48\penalty0 (2):\penalty0
  428--442, 2011.

\bibitem[Doob(2012)]{doob2012measure}
Joseph~L Doob.
\newblock \emph{Measure theory}, volume 143.
\newblock Springer Science \& Business Media, 2012.

\bibitem[Dudley(1967)]{dudley1967coveringNumbers}
Richard~M Dudley.
\newblock The sizes of compact subsets of {H}ilbert space and continuity of
  {G}aussian processes.
\newblock \emph{Journal of Functional Analysis}, 1\penalty0 (3):\penalty0
  290--330, 1967.

\bibitem[Dudley(1969)]{dudley1969speed}
RM~Dudley.
\newblock The speed of mean {G}livenko-{C}antelli convergence.
\newblock \emph{The Annals of Mathematical Statistics}, 40\penalty0
  (1):\penalty0 40--50, 1969.

\bibitem[Eftekhari and Wakin(2017)]{eftekhari2017happens}
Armin Eftekhari and Michael~B Wakin.
\newblock What happens to a manifold under a bi-lipschitz map?
\newblock \emph{Discrete \& Computational Geometry}, 57\penalty0 (3):\penalty0
  641--673, 2017.

\bibitem[Esfahani and Kuhn(2015)]{esfahani2015robustOptimization}
Peyman~Mohajerin Esfahani and Daniel Kuhn.
\newblock Data-driven distributionally robust optimization using the
  {W}asserstein metric: Performance guarantees and tractable reformulations.
\newblock \emph{arXiv preprint arXiv:1505.05116}, 2015.

\bibitem[Fournier and Guillin(2015)]{fournier2015rate}
Nicolas Fournier and Arnaud Guillin.
\newblock On the rate of convergence in {W}asserstein distance of the empirical
  measure.
\newblock \emph{Probability Theory and Related Fields}, 162\penalty0
  (3-4):\penalty0 707--738, 2015.

\bibitem[Gao and Kleywegt(2016)]{gao2016distributionallyRobust}
Rui Gao and Anton~J Kleywegt.
\newblock Distributionally robust stochastic optimization with {W}asserstein
  distance.
\newblock \emph{arXiv preprint arXiv:1604.02199}, 2016.

\bibitem[Garc{\'\i}a et~al.(2015)Garc{\'\i}a, Garc{\'\i}a-R{\'o}denas, and
  G{\'o}mez]{garcia2015functionalKMeans}
Mar{\'\i}a Luz~L{\'o}pez Garc{\'\i}a, Ricardo Garc{\'\i}a-R{\'o}denas, and
  Antonia~Gonz{\'a}lez G{\'o}mez.
\newblock $k$-means algorithms for functional data.
\newblock \emph{Neurocomputing}, 151:\penalty0 231--245, 2015.

\bibitem[Geweke(1989)]{geweke1989bayesian}
John Geweke.
\newblock Bayesian inference in econometric models using monte carlo
  integration.
\newblock \emph{Econometrica: Journal of the Econometric Society}, pages
  1317--1339, 1989.

\bibitem[Gulrajani et~al.(2017)Gulrajani, Ahmed, Arjovsky, Dumoulin, and
  Courville]{gulrajani2017improved}
Ishaan Gulrajani, Faruk Ahmed, Martin Arjovsky, Vincent Dumoulin, and Aaron~C
  Courville.
\newblock Improved training of {W}asserstein {GAN}s.
\newblock In \emph{Advances in Neural Information Processing Systems}, pages
  5769--5779, 2017.

\bibitem[Han et~al.(2015)Han, Jiao, and Weissman]{han2015minimax}
Yanjun Han, Jiantao Jiao, and Tsachy Weissman.
\newblock Minimax estimation of discrete distributions under $\ell_1$ loss.
\newblock \emph{IEEE Transactions on Information Theory}, 61\penalty0
  (11):\penalty0 6343--6354, 2015.

\bibitem[Haussler(1995)]{haussler1995sphere}
David Haussler.
\newblock Sphere packing numbers for subsets of the boolean n-cube with bounded
  {V}apnik-{C}hervonenkis dimension.
\newblock \emph{Journal of Combinatorial Theory, Series A}, 69\penalty0
  (2):\penalty0 217--232, 1995.

\bibitem[Ho et~al.(2017)Ho, Nguyen, Yurochkin, Bui, Huynh, and
  Phung]{ho2017multilevel}
Nhat Ho, XuanLong Nguyen, Mikhail Yurochkin, Hung~Hai Bui, Viet Huynh, and Dinh
  Phung.
\newblock Multilevel clustering via {W}asserstein means.
\newblock In \emph{International Conference on Machine Learning}, pages
  1501--1509, 2017.

\bibitem[Jiao et~al.(2017)Jiao, Han, and Weissman]{jiao2017minimaxL1}
Jiantao Jiao, Yanjun Han, and Tsachy Weissman.
\newblock Minimax estimation of the $l_1$ distance.
\newblock \emph{arXiv preprint arXiv:1705.00807}, 2017.

\bibitem[Johnson et~al.(2005)Johnson, Samworth, et~al.]{johnson2005central}
Oliver Johnson, Richard Samworth, et~al.
\newblock Central limit theorem and convergence to stable laws in {M}allows
  distance.
\newblock \emph{Bernoulli}, 11\penalty0 (5):\penalty0 829--845, 2005.

\bibitem[Kantorovich(1942)]{kantorovich1942translocation}
Leonid~Vitalievich Kantorovich.
\newblock On the translocation of masses.
\newblock In \emph{Dokl. Akad. Nauk. USSR (NS)}, volume~37, pages 199--201,
  1942.

\bibitem[Lei(2018)]{lei2018convergence}
Jing Lei.
\newblock Convergence and concentration of empirical measures under
  {W}asserstein distance in unbounded functional spaces.
\newblock \emph{arXiv preprint arXiv:1804.10556}, 2018.

\bibitem[Leighton and Shor(1989)]{leighton1989tight}
Tom Leighton and Peter Shor.
\newblock Tight bounds for minimax grid matching with applications to the
  average case analysis of algorithms.
\newblock \emph{Combinatorica}, 9\penalty0 (2):\penalty0 161--187, 1989.

\bibitem[Liang(2017)]{liang2017well}
Tengyuan Liang.
\newblock How well can generative adversarial networks ({GAN}) learn densities:
  A nonparametric view.
\newblock \emph{arXiv preprint arXiv:1712.08244}, 2017.

\bibitem[Liang and Stokes(2018)]{liang2018interaction}
Tengyuan Liang and James Stokes.
\newblock Interaction matters: A note on non-asymptotic local convergence of
  generative adversarial networks.
\newblock \emph{arXiv preprint arXiv:1802.06132}, 2018.

\bibitem[Liu et~al.(2018)Liu, Liu, and Lu]{liu2018rate}
Anning Liu, Jian-Guo Liu, and Yulong Lu.
\newblock On the rate of convergence of empirical measure in
  $\infty$-{W}asserstein distance for unbounded density function.
\newblock \emph{arXiv preprint arXiv:1807.08365}, 2018.

\bibitem[Mallat(1999)]{mallat1999wavelet}
St{\'e}phane Mallat.
\newblock \emph{A wavelet tour of signal processing}.
\newblock Academic press, 1999.

\bibitem[Montavon et~al.(2016)Montavon, M{\"u}ller, and
  Cuturi]{montavon2016wassersteinRBMs}
Gr{\'e}goire Montavon, Klaus-Robert M{\"u}ller, and Marco Cuturi.
\newblock {W}asserstein training of restricted {B}oltzmann machines.
\newblock In \emph{Advances in Neural Information Processing Systems}, pages
  3718--3726, 2016.

\bibitem[Morimoto(1963)]{morimoto1963divergences}
Tetsuzo Morimoto.
\newblock Markov processes and the {H}-theorem.
\newblock \emph{Journal of the Physical Society of Japan}, 18\penalty0
  (3):\penalty0 328--331, 1963.

\bibitem[Nagarajan and Kolter(2017)]{nagarajan2017gradient}
Vaishnavh Nagarajan and J~Zico Kolter.
\newblock Gradient descent {GAN} optimization is locally stable.
\newblock \emph{arXiv preprint arXiv:1706.04156}, 2017.

\bibitem[Peyr{\'e}(2011)]{peyre2011numerical}
Gabriel Peyr{\'e}.
\newblock The numerical tours of signal processing.
\newblock \emph{Computing in Science \& Engineering}, 13\penalty0 (4):\penalty0
  94--97, 2011.

\bibitem[Pollard(1982)]{pollard1982quantization}
David Pollard.
\newblock Quantization and the method of $k$-means.
\newblock \emph{IEEE Transactions on Information theory}, 28\penalty0
  (2):\penalty0 199--205, 1982.

\bibitem[Pollard(1990)]{pollard1990empirical}
David Pollard.
\newblock Empirical processes: theory and applications.
\newblock In \emph{NSF-CBMS regional conference series in probability and
  statistics}, pages i--86. JSTOR, 1990.

\bibitem[Reitzner et~al.(2013)Reitzner, Schulte, et~al.]{reitzner2013central}
Matthias Reitzner, Matthias Schulte, et~al.
\newblock Central limit theorems for {$U$}-statistics of {P}oisson point
  processes.
\newblock \emph{The Annals of Probability}, 41\penalty0 (6):\penalty0
  3879--3909, 2013.

\bibitem[Rio et~al.(2009)]{rio2009upper}
Emmanuel Rio et~al.
\newblock Upper bounds for minimal distances in the central limit theorem.
\newblock \emph{Annales de l'Institut Henri Poincar{\'e}, Probabilit{\'e}s et
  Statistiques}, 45\penalty0 (3):\penalty0 802--817, 2009.

\bibitem[Rio et~al.(2011)]{rio2011asymptotic}
Emmanuel Rio et~al.
\newblock Asymptotic constants for minimal distance in the central limit
  theorem.
\newblock \emph{Electronic Communications in Probability}, 16:\penalty0
  96--103, 2011.

\bibitem[R{\"u}schendorf(1985)]{ruschendorf1985wasserstein}
Ludger R{\"u}schendorf.
\newblock The {W}asserstein distance and approximation theorems.
\newblock \emph{Probability Theory and Related Fields}, 70\penalty0
  (1):\penalty0 117--129, 1985.

\bibitem[Sadhanala et~al.(2016)Sadhanala, Wang, and
  Tibshirani]{sadhanala2016totalVariation}
Veeranjaneyulu Sadhanala, Yu-Xiang Wang, and Ryan~J Tibshirani.
\newblock Total variation classes beyond 1d: Minimax rates, and the limitations
  of linear smoothers.
\newblock In \emph{Advances in Neural Information Processing Systems}, pages
  3513--3521, 2016.

\bibitem[Shor et~al.(1991)Shor, Yukich, et~al.]{shor1991minimax}
Peter~W Shor, Joseph~E Yukich, et~al.
\newblock Minimax grid matching and empirical measures.
\newblock \emph{The Annals of Probability}, 19\penalty0 (3):\penalty0
  1338--1348, 1991.

\bibitem[Singh et~al.(2018)Singh, Uppal, Li, Li, Zaheer, and
  P{\'o}czos]{singh2018adversarial}
Shashank Singh, Ananya Uppal, Boyue Li, Chun-Liang Li, Manzil Zaheer, and
  Barnab{\'a}s P{\'o}czos.
\newblock Nonparametric density estimation under adversarial losses.
\newblock In \emph{Advances in Neural Information Processing Systems}, pages
  10225--10236, 2018.

\bibitem[Sriperumbudur et~al.(2010)Sriperumbudur, Fukumizu, Gretton,
  Sch{\"o}lkopf, and Lanckriet]{sriperumbudur2010integralProbabilityMetrics}
Bharath~K Sriperumbudur, Kenji Fukumizu, Arthur Gretton, Bernhard
  Sch{\"o}lkopf, and Gert~RG Lanckriet.
\newblock Non-parametric estimation of integral probability metrics.
\newblock In \emph{Information Theory Proceedings (ISIT), 2010 IEEE
  International Symposium on}, pages 1428--1432. IEEE, 2010.

\bibitem[Sriperumbudur et~al.(2012)Sriperumbudur, Fukumizu, Gretton,
  Sch{\"o}lkopf, Lanckriet, et~al.]{sriperumbudur2012empirical}
Bharath~K Sriperumbudur, Kenji Fukumizu, Arthur Gretton, Bernhard
  Sch{\"o}lkopf, Gert~RG Lanckriet, et~al.
\newblock On the empirical estimation of integral probability metrics.
\newblock \emph{Electronic Journal of Statistics}, 6:\penalty0 1550--1599,
  2012.

\bibitem[Trillos and Slep{\v{c}}ev(2015)]{trillos2015rate}
Nicol{\'a}s~Garcia Trillos and Dejan Slep{\v{c}}ev.
\newblock On the rate of convergence of empirical measures in
  $\infty$-transportation distance.
\newblock \emph{Canadian Journal of Mathematics}, 67\penalty0 (6):\penalty0
  1358--1383, 2015.

\bibitem[Trillos and Slep{\v{c}}ev(2018)]{trillos2018variational}
Nicolas~Garcia Trillos and Dejan Slep{\v{c}}ev.
\newblock A variational approach to the consistency of spectral clustering.
\newblock \emph{Applied and Computational Harmonic Analysis}, 45\penalty0
  (2):\penalty0 239--281, 2018.

\bibitem[Tseng(2000)]{tseng2000kFlats}
Paul Tseng.
\newblock Nearest $q$-flat to $m$ points.
\newblock \emph{Journal of Optimization Theory and Applications}, 105\penalty0
  (1):\penalty0 249--252, 2000.

\bibitem[Tsybakov(2009)]{tsybakov2009introduction}
Alexandre~B Tsybakov.
\newblock \emph{Introduction to nonparametric estimation}.
\newblock Springer Series in Statistics. Springer, New York, 2009.

\bibitem[Uppal et~al.(2019)Uppal, Singh, and
  P{\'o}czos]{uppal2019nonparametric}
Ananya Uppal, Shashank Singh, and Barna{\'a}s P{\'o}czos.
\newblock Nonparametric density estimation under besov ipm losses.
\newblock \emph{arXiv preprint arXiv:1902.03511}, 2019.

\bibitem[Villani(2008)]{villani2008optimalTransport}
C{\'e}dric Villani.
\newblock \emph{Optimal transport: old and new}, volume 338.
\newblock Springer Science \& Business Media, 2008.

\bibitem[Weed and Bach(2017)]{weed2017sharp}
Jonathan Weed and Francis Bach.
\newblock Sharp asymptotic and finite-sample rates of convergence of empirical
  measures in {W}asserstein distance.
\newblock \emph{arXiv preprint arXiv:1707.00087}, 2017.

\bibitem[Zhang(2002)]{zhang2002covering}
Tong Zhang.
\newblock Covering number bounds of certain regularized linear function
  classes.
\newblock \emph{Journal of Machine Learning Research}, 2\penalty0
  (Mar):\penalty0 527--550, 2002.

\bibitem[Zhou(2002)]{zhou2002covering}
Ding-Xuan Zhou.
\newblock The covering number in learning theory.
\newblock \emph{Journal of Complexity}, 18\penalty0 (3):\penalty0 739--767,
  2002.

\end{thebibliography}

\newpage
\appendix
\section{Preliminary Lemmas and Proof Sketch of Upper Bound}
\label{sec:lemmas}

In this section, we outline the proof of Theorem~\ref{thm:unbounded_upper_bound}, our main upper bound result.
We begin by providing a few basic lemmas; these lemmas are not fundamentally novel, but they will be used in the subsequent proofs of our main upper and lower bounds, and also help provide intuition for the behavior of the Wasserstein metric and its connections to other metrics between probability distributions. The proofs of these lemmas are given later, in Appendix~\ref{app:proofs}. Our first lemma relates Wasserstein distance to the notion of resolution of a partition.

\begin{lemma}
Suppose $\S \in \SS$ is a countable Borel partition of $\Omega$. Let $P$ and $Q$ be Borel probability measures such that, for every $S \in \S$, $P(S) = Q(S)$. Then, for any $r \geq 1$, $W_r(P, Q) \leq \Res(\S)$.
\label{lemma:measures_agree_on_partition}
\end{lemma}

Our next lemma gives simple lower and upper bounds on the Wasserstein distance between distributions supported on a countable subset $\X \subseteq \Omega$, in terms of $\Diam(\X)$ and $\Sep(\X)$. Since our main results will utilize coverings and packings to approximate $\Omega$ by finite sets, this lemma will provide a first step towards approximating (in Wasserstein distance) distributions on $\Omega$ by distributions on these finite sets. Indeed, the lower bound in Inequality~\eqref{ineq:countable_support_bound} will suffice to prove our lower bounds, although a tighter upper bound, based on the upper bound in~\eqref{ineq:countable_support_bound}, will be necessary to obtain tight upper bounds.

\begin{lemma}
Suppose $(\Omega, \rho)$ is a metric space, and suppose $P$ and $Q$ are Borel probability distributions on $\Omega$ with countable support; i.e., there exists a countable set $\X \subseteq \Omega$ with $P(\X) = Q(\X) = 1$. Then, for any $r \geq 1$,
\begin{equation}
(\Sep(\X))^r \sum_{x \in \X} \left| P(\{x\}) - Q(\{x\}) \right|
  \leq W_r^r(P,Q)
  \leq (\Diam(\X))^r \sum_{x \in \X} \left| P(\{x\}) - Q(\{x\}) \right|.
\label{ineq:countable_support_bound}
\end{equation}
\label{lemma:countable_support_bound}
\end{lemma}

\begin{remark}
Recall that the term $\sum_{x \in \X} \left| P(\{x\}) - Q(\{x\}) \right|$
in Inequality~\eqref{ineq:countable_support_bound} is the $\L_1$ distance
\[\|p - q\|_1 := \sum_{x \in \X} \left| p(x) - q(x) \right|\]
between the densities $p$ and $q$ of $P$ and $Q$ with respect to the counting measure on $\X$, and that this same quantity is twice the total variation distance \citep{villani2008optimalTransport}
\[TV(P,Q) := \sup_{A \subseteq \Omega} \left| P(A) - Q(A) \right|.\]
Hence, Lemma~\ref{lemma:countable_support_bound} can be equivalently written as
\[\Sep(\Omega) \left( \|p - q\|_1 \right)^{1/r} \leq W_r(P,Q) \leq \Diam(\Omega) \left( \|p - q\|_1 \right)^{1/r}\]
and as
\[\Sep(\Omega) \left( 2 TV(P,Q) \right)^{1/r} \leq W_r(P,Q) \leq \Diam(\Omega) \left( 2 TV(P,Q) \right)^{1/r},\]
bounding the $r$-Wasserstein distance in terms of the $\L_1$ and total variation distance.
As noted in Example~\ref{ex:discrete_bound}, equality holds in \eqref{ineq:countable_support_bound} precisely when $\rho$ is the unit discrete metric given by $\rho(x,y) = 1_{\{x \neq y\}}$ for all $x,y \in \Omega$.

On metric spaces that are discrete (i.e., when $\Sep(\Omega) > 0$), the Wasserstein metric is (topologically) at least as strong as the total variation metric (and the $\L_1$ metric, when it is well-defined), in that convergence in Wasserstein metric implies convergence in total variation (and $\L_1$, respectively). On the other hand, on bounded metric spaces, the converse is true. In either of these cases, \emph{rates} of convergence may differ between metrics, although, in metric spaces that are both discrete \textit{and} bounded (e.g., any finite space), we have $W_r \asymp TV^{1/r}$.
\label{remark:Wasserstein_L1_TV}
\end{remark}

To obtain tight bounds as discussed below, we will require not only a partition of the sample space $\Omega$, but a nested sequence of partitions, defined as follows.

\begin{definition}[Refinement of a Partition, Nested Partitions]
Suppose $\S, \T \in \SS$ are partitions of $\Omega$. $\T$ is said to be a \emph{refinement of $\S$} if, for every $T \in \T$, there exists $S \in \S$ with $T \subseteq S$. A sequence $\{\S_k\}_{k \in \N}$ of partitions is called \emph{nested} if, for each $k \in \N$, $\S_k$ is a refinement of $\S_{k + 1}$,
\end{definition}

While Lemma~\ref{lemma:countable_support_bound} gave a simple upper bound on the Wasserstein distance, the factor of $\Diam(\Omega)$ turns out to be too large to obtain tight rates for a number of cases of interest (such as the $D$-dimensional unit cube $\Omega = [0,1]^D$, discussed in Example~\ref{ex:unit_cube_lower_bound}). The following lemma gives a tighter upper bound, based on a hierarchy of nested partitions of $\Omega$; this allows us to obtain tighter bounds (than $\Diam(\Omega)$) on the distance that mass must be transported between $P$ and $Q$.
Note that, when $K = 1$, Lemma~\ref{lemma:nested_partitions_Wasserstein_bound} reduces to a trivial combination of Lemmas~\ref{lemma:measures_agree_on_partition} and \ref{lemma:countable_support_bound}; indeed, these lemmas are the starting point for proving Lemma~\ref{lemma:nested_partitions_Wasserstein_bound} by induction on $K$.

Note that the idea of such a ``multi-resolution'' upper bound has been utilized extensively before, and numerous versions have been proven before (see, e.g., Fact 6 of \citet{do2011sublinearTimeEMD}, Lemma 6 of \citet{fournier2015rate}, or Proposition 1 of \citet{weed2017sharp}). Most of these versions have been specific to Euclidean space; to the best of our knowledge, only Proposition 1 of \citet{weed2017sharp} applies to general metric spaces. However, that result also requires that $(\Omega,\rho)$ is totally bounded (more precisely, that $m_x^\infty(P) < \infty$, for some $x \in \Omega$).

\begin{lemma}
Let $K$ be a positive integer. Suppose that $\Omega$ has finite diameter $\Delta := \Diam(\Omega) < \infty$, and let $\{\S_k\}_{k \in \N}$ be a nested sequence of countable Borel $\delta$-partitions of $(\Omega,\rho)$. Then, for any $r \geq 1$ and Borel probability measures $P$ and $Q$ on $\Omega$,
\begin{equation}
W_r^r(P, Q)
  \leq \Delta^r \left( (\Res(\S_0))^r + \sum_{k = 1}^K \left( \Res(\S_k) \right)^r
                                      \left( \sum_{S \in \S_{k + 1}} \left| P(S) - Q(S) \right| \right) \right).
\label{ineq:multiresolution_bound}
\end{equation}
\label{lemma:nested_partitions_Wasserstein_bound}
\end{lemma}

Lemma~\ref{lemma:nested_partitions_Wasserstein_bound} requires a sequence of partitions of $\Omega$ that is not only multi-resolution but also nested. While the $\epsilon$-covering number implies the existence of small partitions with small resolution, these partitions need not be nested as $\epsilon$ becomes small. For this reason, we now give a technical lemma that, given any sequence of partitions, constructs a \textit{nested} sequence of partitions of the same cardinality, with only a small increase in resolution.

\begin{lemma}
Suppose $\S$ and $\T$ are partitions of $(\Omega,\rho)$, and suppose $\S$ is countable. Then, there exists a partition $\S'$ of $(\Omega,\rho)$ such that:
\begin{enumerate}[label=\alph*)]
\item
$|\S'| \leq |\S|$.
\item
$\Res(\S') \leq \Res(\S) + 2\Res(\T)$.
\item
$\T$ is a refinement of $\S'$.
\end{enumerate}
\label{lemma:fine_refinement}
\end{lemma}

Lemmas~\ref{lemma:nested_partitions_Wasserstein_bound} and \ref{lemma:fine_refinement} are the main tools needed to bound the expected Wasserstein distance $\E[W_r^r(P, \hat P)]$ of the empirical distribution from the true distribution into a sum of its expected errors on each element of a nested partition of $\Omega$. Then, we will need to control the total expected error across these partition elements, which we will show behaves similarly to the $\L_1$ error of the standard maximum likelihood (mean) estimator a multinomial distribution from its true mean. Thus, the following result of \citet{han2015minimax} will be useful.

\begin{lemma}[Theorem 1 of \citep{han2015minimax}]
Suppose $(X_1,...,X_K) \sim \operatorname{Multinomial}(n,p_1,...,p_K)$. Let
\[Z := \|X - n p\|_1 = \sum_{k = 1}^K \left| X_k - n p_k \right|.\] Then,
$\E \left[ Z/n \right] \leq \sqrt{(K - 1)/n}$.
\label{lemma:multinomial_expectation}
\end{lemma}

Finally, we are ready to prove Theorem~\ref{thm:expectation_bound_appendix}.

\begin{theorem}
Let $(\Omega,\rho)$ be a metric space on which $P$ is a Borel probability measure. Suppose $\Omega$ has finite diameter $\Delta := \Diam(\Omega) < \infty$. Let $\hat P$ denote the empirical distribution of $n$ IID samples $X_1,...,X_n \IID P$, given by
\[\hat P(S) := \frac{1}{n} \sum_{i = 1}^n 1_{\{X_i \in S\}}, \quad \forall S \in \Sigma.\]
Then, for any sequence $\{\epsilon_k\}_{k \in [K]} \in (0,\infty)^K$,
\[\E \left[ W_r^r(P, \hat P) \right] \leq \Delta^r \left( \epsilon_K^r + \frac{1}{\sqrt{n}} \sum_{k = 1}^K \left( \sum_{j = k - 1}^K 2^{j - k} \epsilon_j \right)^r \sqrt{N(\Omega, \epsilon_k) - 1} \right).\]
\label{thm:expectation_bound_appendix}
\end{theorem}

\begin{proof}
By recursively applying Lemma~\ref{lemma:fine_refinement}, there exists a sequence $\{\S_k\}_{k \in [K]}$ of partitions of $(\Omega,\rho)$ satisfying the following conditions:
\begin{enumerate}
\item
for each $k \in [K]$, $|\S_k| = N(\epsilon_k)$.
\item
for each $k \in [K]$,
$\displaystyle \Res(\S_k) \leq \sum_{j = k}^K 2^{j - k} \epsilon_j$.
\item
$\{S_k\}_{k \in [K]}$ is nested.
\end{enumerate}
Note that, for any $k \in [K]$, the vector $n\hat P(S)$ (indexed by $S \in \S_k$) follows an $n$-multinomial distribution over $|\S_k|$ categories, with means given by $P(S)$; i.e., \[(n\hat P(S_1),...,n\hat P(S_k)) \sim \operatorname{Multinomial}(n,P(S_1),...,P(S_k)).\]
Thus, by Lemma~\ref{lemma:multinomial_expectation}, for each $k \in [K]$,
\[\E \left[ \sum_{S \in \S_k} \left| P(S) - \hat P(S) \right| \right]
  \leq \sqrt{\frac{|\S_k| - 1}{n}}
  = \sqrt{\frac{N(\epsilon_k) - 1}{n}}.\]
Thus, by Lemma~\ref{lemma:nested_partitions_Wasserstein_bound},
\begin{align*}
\E \left[ W_r^r(P, \hat P) \right]
& \leq \Delta^r \E \left[ \epsilon_K^r + \sum_{k = 1}^K \left( \sum_{j = k}^K 2^{j - k} \epsilon_j \right)^r \left( \sum_{S \in \S_k} \left| P(S) - \hat P(S) \right| \right) \right] \\
& \leq \Delta^r \left( \epsilon_K^r + \sum_{k = 1}^K \left( \sum_{j = k}^K 2^{j - k} \epsilon_j \right)^r \E \left[ \sum_{S \in \S_k} \left| P(S) - \hat P(S) \right| \right] \right) \\
& \leq \Delta^r \left( \epsilon_K^r + \frac{1}{\sqrt{n}} \sum_{k = 1}^K \left( \sum_{j = k}^K 2^{j - k} \epsilon_j \right)^r \sqrt{N(\epsilon_k) - 1} \right).
\end{align*}
\end{proof}

\section{Proof of Upper Bound (Theorem~\ref{thm:unbounded_upper_bound})}
\label{sec:detailed_proof_of_upper_bound}

In this section, we prove our more general upper bound, Theorem~\ref{thm:unbounded_upper_bound_appendix}, which applies to potentially unbounded metric spaces $(\Omega,\rho)$, assuming that $P$ is sufficiently concentrated (i.e., has at least $\ell > 0$ finite moments).

The basic idea is to partition the potentially unbounded metric space $(\Omega,\rho)$ into countably many totally bounded subsets $B_1,B_2,...$, and to decompose the Wasserstein error into its error on each $B_i$, weighted by the probability $P(B_i)$. Specifically, fixing an arbitrary base point $x_0$, $B_1,B_2,...$ will be spherical shells, such that $x_0 \in B_1$, and both the distance between $B_i$ and $x_0$, as well as the size (covering number) of $B_i$, increase with $i$. For large $i$, the assumption that $P$ has $\ell$ bounded moments implies (by a simple application of Markov's inequality in Eq. \ref{eq:markov_application}) that $P(B_i)$ is small, whereas, for small $i$, we adapt our previous result Theorem~\ref{thm:expectation_bound_appendix} in terms of the covering number.

To carry out this approach, we will need two new lemmas. The first decomposes Wasserstein distance into the sum of its distances on each $B_i$, and can be considered an adaptation of Lemma 2.2 of \citet{lei2018convergence} from Banach spaces to general metric spaces.

\begin{lemma}
Fix a reference point $x_0 \in \Omega$ and a non-decreasing real-valued sequence $\{w_k\}_{k \in \N}$ with $w_0 = 0$ and $\lim_{k \to \infty} w_k = \infty$. For each $k \in \N$, define
\[B_k := \left\{x \in \Omega : w_k \leq \rho(x_0, x) < w_{k + 1} \right\}.\]
Then, for any Borel probability measures $P$ and $Q$ on $\Omega$,
\[W_r^r(P,Q)
  \leq 2^r \sum_{k = 0}^\infty
    w_k^r
    \min \left\{ P(B_k), Q(B_k) \right\}
    W_r^r(P_{B_k},Q_{B_k})
  + 2 w_k^r \left| P(B_k) - Q(B_k) \right|.\]
where, for any sets $A, B \subseteq \Omega$,
\[P_A(B) = \frac{P(A \cap B)}{P(B)}\]
(under the convention that $\frac{0}{0} = 0$) denotes the conditional probability of $B$ given $A$, under $P$.
\label{lemma:sigma_finite_partition}
\end{lemma}

The second lemma is more nuanced variant of Lemma~\ref{lemma:multinomial_expectation} (albeit, leading to slightly looser constants). When $i$ is large the covering number of $B_i$ can become quite large, but the total probability $P(B_i)$ is quite small. Whereas Lemma~\ref{lemma:multinomial_expectation} depends only on the size of the partition, the following result will allow us to control the total error using both of these factors.

\begin{lemma}[Theorem 1 of \citet{berend2013binomialMAD}]
Suppose $X \sim \operatorname{Binomial}(n,p)$. Then, we have the bound
\begin{equation}
\E \left[ \left| X - n p \right| \right] \leq n \min \left\{ 2P(A), \sqrt{P(A)/n} \right\}.
\label{ineq:binomial_MAD}
\end{equation}
on the mean absolute deviation of $X$.
\label{lemma:binomial_MAD}
\end{lemma}

Finally, we are ready to prove our main upper bound result (Theorem~\ref{thm:unbounded_upper_bound}) for unbounded metric spaces. We prove a slightly more general version, with full sequences $\{w_k\}_{k \in \N}$ and $\{\epsilon_j\}_{j \in \N}$ of free parameters; the version of Theorem~\ref{thm:unbounded_upper_bound} stated in the main paper follows by setting $w_k = 2^k$ and $\epsilon_j = 4^{-j}$, which seems sufficient to obtain tight rates in most contexts.

\begin{customthm}{\ref{thm:unbounded_upper_bound}}[General Upper Bound for Unbounded Metric Spaces]
Let $x_0 \in \Omega$ and suppose $m_{\ell,x_0}(P) \in [1, \infty)$. Let $J$ be a positive integer. Fix two non-decreasing real-valued sequences $\{w_k\}_{k \in \N}$ and $\{\epsilon_j\}_{j \in \N}$, of which $\{w_k\}_{k \in \N}$ is non-decreasing with $w_0 = 0$ and $\lim_{k \to \infty} w_k = \infty$ and $\{\epsilon_j\}_{j \in [J]}$ is non-increasing.
For each positive integer $k$, define
\[B_k(x_0) := \left\{ y \in \Omega : w_{k - 1} \leq \rho(x_0, x) < w_k \right\}.\]
Then,
\begin{align*}
\E \left[ W_r^r(P, \hat P) \right]
& \leq 2^r m_{\ell,x_0}^\ell \sum_{k \in \N} w_k^{r - \ell/2} \min \left\{ 2w_k^{-\ell/2}, \sqrt{\frac{1}{n}} \right\} + w_k^{r-\ell} \epsilon_J^r \\
& \hspace{2cm} + w_k^r \sum_{j = 1}^J \left( \sum_{t = j}^J 2^{J - t} \epsilon_t \right)^r \min \left\{ 2w_k^{-\ell}, \sqrt{\frac{w_k^{-\ell}}{n} N(B_k,w_k \epsilon_j)} \right\}.
\end{align*}
\label{thm:unbounded_upper_bound_appendix}
\end{customthm}

\begin{proof}
As in the proof of Theorem~\ref{thm:expectation_bound_appendix}, by recursively applying Lemma~\ref{lemma:fine_refinement}, for each $k \in \N$, we can construct a nested sequence $\{\S_{k,j}\}_{j \in [J]}$ of partitions of $B_k$ such that, for each $j \in [J]$,
\begin{equation}
|\S_{k,j}| = N(B_k,w_k \epsilon_j)
  \quad \text{ and } \quad
  \Res(\S_{k,j}) \leq w_k \sum_{t = 0}^j 2^t \epsilon_t.
  \label{eq:recursive_fine_refinement}
\end{equation}
Since each $P_{B_k}$ and $\hat P_{B_k}$ are supported only on $B_k$, plugging the bound Lemma~\ref{lemma:nested_partitions_Wasserstein_bound} into the bound in Lemma~\ref{lemma:sigma_finite_partition} gives
\begin{align*}
& W_r^r(P, \hat P) \\
& \leq 2^r \sum_{k \in \N} w_k^r \min \left\{ P(B_k), \hat P(B_k) \right\} \left( \left( \Res(\S_{k,0}) \right)^r + \sum_{j = 1}^J \left( \Res(\S_{k,j}) \right)^r \sum_{S \in \S_{k,j + 1}} \left| P_{B_k}(S) - \hat P_{B_k}(S) \right| \right) \\
& \hspace{1cm} + w_k^r \left| P(B_k) - \hat P(B_k) \right| \\
& \leq 2^r \sum_{k \in \N} w_k^r P(B_k) \left( \left( \Res(\S_{k,0}) \right)^r + \sum_{j = 1}^J \left( \Res(\S_{k,j}) \right)^r \sum_{S \in \S_{k,j + 1}} \left| P_{B_k}(S) - \hat P_{B_k}(S) \right| \right) \\
& \hspace{1cm} + w_k^r \left| P(B_k) - \hat P(B_k) \right|
\end{align*}
Since each $\hat P(S) \sim \operatorname{Binomial}(n, P(S))$, for each $k \in \N$ and $j \in [J]$, Lemma~\ref{lemma:binomial_MAD} followed by Cauchy-Schwarz gives
\begin{align*}
\E \left[ \sum_{S \in \S_{k,j}} \left| P(S) - \hat P(S) \right| \right]
& \leq \sum_{S \in \S_{k,j + 1}} \min \left\{ 2P(S), \sqrt{P(S)/n} \right\} \\
& \leq \min \left\{ 2P(B_k), \sqrt{\frac{P(B_k)}{n} |\S_{k,j}|} \right\}.
\end{align*}
Therefore, taking expectations (over $X_1,...,X_n$), applying Inequality~\ref{eq:recursive_fine_refinement}, and applying Lemma~\ref{lemma:binomial_MAD} once more gives
\begin{align*}
\E \left[ W_r^r(P, \hat P) \right]
& \leq 2^r \sum_{k \in \N} w_k^r \min \left\{ 2P(B_k), \sqrt{P(B_k)/n} \right\} + P(B_k) w_k^r \epsilon_J^r \\
& \hspace{1cm} + w_k^r \sum_{j = 1}^J \left( \sum_{t = 0}^j 2^t \epsilon_j \right)^r \min \left\{ 2P(B_k), \sqrt{\frac{P(B_k)}{n} N(B_k,w_k \epsilon_{j + 1})} \right\},
\end{align*}
where we used Lebesgue monotone convergence theorem to interchange the expectation and the infinite summation over $k$.
Now note that, by Markov's inequality,
\begin{equation}
P(B_k)
  \leq \pr_{X \sim P} \left[ \rho(x_0, X) \geq w_k \right]
  = \pr_{X \sim P} \left[ \rho^\ell (x_0, X) \geq w_k^\ell \right]
  \leq \frac{m_{\ell,x_0}^\ell(P)}{w_k^\ell}.
  \label{eq:markov_application}
\end{equation}
Therefore, since assumed that $m_{\ell,x_0}^\ell \geq 1$ (and hence $m_{\ell,x_0}^\ell \geq m_{\ell,x_0}^{\ell/2}$),
\begin{align*}
\E \left[ W_r^r(P, \hat P) \right]
& \leq m_{\ell,x_0}^\ell 2^r \sum_{k \in \N} w_k^r \min \left\{ 2w_k^{-\ell}, \sqrt{w_k^{-\ell}/n} \right\} + w_k^{r-\ell} \epsilon_J^r \\
& \hspace{1cm} + w_k^r \sum_{j = 1}^J \left( \sum_{t = 0}^j 2^t \epsilon_j \right)^r \min \left\{ 2w_k^{-\ell}, \sqrt{\frac{w_k^{-\ell}}{n} N(B_k,w_k \epsilon_{j + 1})} \right\},
\end{align*}
proving the theorem.
\end{proof}

\section{Proofs of Minimax Lower Bound in terms of the Packing Radius}
\label{sec:lower_bound_proof_packing_radius}

In this section, we provide a proof of our main lower bound, Theorem~\ref{thm:Wasserstein_distribution_estimation_lower_bound} in the main text. The proof consists of two main steps. First, we show that the minimax error of estimation in Wasserstein distance can be lower bounded by a product of two terms, one depending on the packing radius $R$ and the other depending on the minimax risk of estimating a particular discrete (i.e., multinomial) distribution under $\L_1$ loss. The second step is then to apply a minimax lower bound on the risk of estimating a multinomial distribution under $\L_1$ loss. These two steps respectively rely on two lemmas, Lemma~\ref{lemma:wasserstein_projections} and Lemma~\ref{lemma:multinomial_minimax_lower_bound} given below.

The first lemma implies that, when a distribution $P$ is supported on a finite subset $\D$ of the sample space, then there exists an estimator $\hat P_\D$ of $\hat P$ that is supported on $\D$ is minimax optimal, up to a small constant factor. While this fact is relatively obvious for measure-theoretic metrics such as $\L_p$ distances, it is somewhat less obvious for Wasserstein distances, which also depend on metric properties of the space. This observation is key to lower bounding the minimax rate in terms of the minimax rate for estimating a discrete distribution.
\begin{lemma}[Wasserstein Projections]
Let $(\X,\rho)$ be a metric space and let $\D \subseteq \X$ be finite. Let $\P$ denote the family of all Borel probability distributions on $\X$, and let
\[\P_\D := \{P \in \P : P(\D) = 1\}\]
denote the set of distributions supported only on $\D$. Suppose $P \in \P_\D$ and $Q \in \P$. Then,
\[\argmin_{\tilde Q \in \P_\D} W_r(Q, \tilde Q) \neq \emptyset
  \quad \text{ and, for any } \quad
  Q' \in \argmin_{\tilde Q \in \P'} W_r(Q, \tilde Q),\]
we have $W_r(P, Q') \leq 2W_r(P, Q)$.
  \label{lemma:wasserstein_projections}
\end{lemma}

\begin{proof}
Let $\{\S_x\}_{x \in \D}$ denote the Voronoi diagram of $\X$ with respect to $\D$; that is, for each $x \in \D$, let
\[\S_x := \{y \in \X : x \in \argmin_{z \in \D} \rho(x,y) \}.\]
Since $\{S_x\}_{x \in \D}$ is a finite cover of $\X$, we can disjointify it (see Remark~\ref{remark:disjointification}) while retaining the property that, for every $x \in \D$ and every $y \in \S_x$, $\rho(x,y) = \min_{z \in \D} \rho(z,y)$; hence, we assume without loss of generality that $\{\S_x\}_{x \in \D}$ is a partition of $\X$. Then, there is a unique distribution $Q' \in \P_D$ such that, for each $x \in \D$, $Q'(\{x\}) = Q(\S_x)$. It is easy to see by definition of the Voronoi diagram that $Q' \in \argmin_{\tilde Q \in \P_\D} W_r(Q, \tilde Q)$; the unique transportation map $\mu_* \in \Pi(Q,Q')$ such that each $\mu(\S_x,\{x\}) = Q(\S_x)$ clearly minimizes
\[\E_{(X,Y) \sim \mu} \left[ \rho^r(X,Y) \right]\]
over all $\mu \in \bigcup_{\tilde Q \in \P_\D} \Pi(Q, \tilde Q)$. Moreover, since $P \in \P_D$, by the triangle inequality and the definition of $Q'$, $W_r(P, Q') \leq W_r(P, Q) + W_r(Q, Q') \leq 2 W_r(P, Q)$.
\end{proof}

The second lemma is a simple minimax lower bound for the problem of estimating the mean vector of a multinomial distribution, under $\L_1$ loss.

\begin{lemma}[Minimax Lower Bound for Mean of Multinomial Distribution]
Suppose $k \leq 32 n$. Let $p \in \Delta^k$, and suppose $X_1,...,X_n \IID \operatorname{Categorical}(p_1,...,p_k)$ are distributed IID according to a categorical distribution on $[k]$, with mean vector $p$. Then, we have the following minimax lower bound for estimating $p$ under $\L^1$-loss:
\[\inf_{\hat p} \sup_{p \in \Delta^k} \E \left[ \|p - \hat p\|_1 \right] \geq \frac{3\log 2}{4096} \sqrt{\frac{k - 1}{n}},\]
where the infimum is taken over all estimators (i.e., all (potentially randomized) functions $\hat p : [k]^n \to \Delta^k$ of the data).
\label{lemma:multinomial_minimax_lower_bound}
\end{lemma}

Note that, while the above result is phrased for categorical distributions to simplify notation in the proof, the result is equivalent to a statement for multinomial distributions, since $\sum_{i = 1}^n X_i \sim \text{Multinomial}(n,p_1,...,p_k)$ and $X_1,...,X_n$ are assumed to be IID and therefore exchangeable.

\begin{proof}
We follow a standard procedure for proving minimax lower bounds based on Fano's inequality, as outlined in Section 2.6 of \citet{tsybakov2009introduction}.

Let $p_0 = \left( 1/k, ...., 1/k \right) \in \Delta^K$ denote the uniform vector in $\Delta^k$.
Let $\I := \left[ \lfloor \frac{k}{2} \rfloor \right]$. For each $j \in \I$, define $\phi_j : [k] \to \R^k$ by
\[\phi_j := 1_{\{2j - 1\}} - 1_{\{2j\}},\]
and, for each $\tau \in \{-1,1\}^\I$, define
\[p_\tau := p_0 + \frac{c}{k} \sum_{j \in \I} \tau_j \phi_j,\]
where
\[c = \frac{1}{16} \sqrt{\frac{k - 1}{n}\log 2} \leq \frac{1}{2}.\]
Note that, since $|c| \leq 1$ and, for each $j \in \I$, $\sum_{\ell \in [k]} \phi_j(\ell) = 0$, each $p_\tau \in \Delta^k$.
Observe that, for any $\tau,\tau' \in \{-1,1\}^\I$, we have
\[\|p_\tau - p_{\tau'}\|_1
  = \frac{4 c \omega(\tau,\tau')}{k},
  \quad \text{ where } \quad
  \omega(\tau,\tau') = \sum_{i \in \I} 1_{\{\tau_i \neq \tau_i'\}}\]
denotes the Hamming distance between $\tau$ and $\tau'$. By the Varshamov-Gilbert bound (see, e.g., Lemma 2.9 of \citet{tsybakov2009introduction}), there exists a subset $T \subseteq \{-1,1\}^\I$ such that $\log |T| \geq \frac{\lfloor k/2 \rfloor \log 2}{8}$ and, for every $\tau, \tau' \in T$,
\[\omega(\tau,\tau') \geq \frac{|\I|}{8}
  = \frac{\lfloor k/2 \rfloor}{8},
  \quad \text{ and hence } \quad
  \|p_\tau - p_{\tau'}\|_1 \geq c \frac{\lfloor k/2 \rfloor}{2k}.\]
Also, for any $\tau \in T$,
\begin{align*}
D_{KL}(p_\tau^n,p_0^n)
& = n D_{KL}(p_\tau,p_0) \\
& = n \sum_{j \in [k]} p_{\tau,j} \log \left( \frac{p_{\tau,j}}{p_{0,j}}\right) \\
& = n \sum_{j \in \I} p_{\tau,2j - 1} \log \left( \frac{p_{\tau,2j - 1}}{1/k} \right) + p_{\tau,2j} \log \left( \frac{p_{\tau,2j}}{1/k} \right) \\
& = \frac{n |\I|}{k} \left( (1 - c) \log \left( 1 - c \right) + (1 + c) \log \left( 1 + c \right) \right)
\end{align*}
One can check (e.g., by Taylor expansion) that, for any $c \in (0,1/2)$,
\[(1 - c) \log \left( 1 - c \right) + (1 + c) \log \left( 1 + c \right)
  < 2c^2.\]
Thus, since $|\I| \leq k/2$,
\[D_{KL}(p_\tau^n,p_0^n)
  \leq \frac{2 n |\I| c^2}{k}
  \leq n c^2.\]
It follows that from the choice of $c$ (and noting that, by the assumptions that $k \leq 32n$, $c \in (0,1/2)$) that
\[\frac{1}{|T|} \sum_{\tau \in T} D_{KL}(p_\tau^n, p_0^n)
  \leq nc^2
  \leq \frac{\lfloor k/2 \rfloor \log 2}{128}
  \leq \frac{1}{16} \log |T|.\]
Therefore, by Fano's method for lower bounds (see, e.g., Theorem 2.5 of \citet{tsybakov2009introduction}, with $\alpha = 1/16$ and
\[s := \frac{c}{16}
    \leq c \frac{\lfloor k/2 \rfloor}{4k}
    \leq \frac{1}{2} \|p_\tau - p_{\tau'}\|_1,\]
we have
\begin{align*}
\inf_{\hat p} \sup_{p \in \Delta^k} \E \left[ \|p - \hat p\|_1 \right]
& \geq \inf_{\hat p} \sup_{p \in \Delta^k} c \frac{\lfloor k/2 \rfloor}{4k} \pr \left[ \|p - \hat p\|_1 \geq c \frac{\lfloor k/2 \rfloor}{4k} \right] \\
& \geq c \frac{\lfloor k/2 \rfloor}{4k} \frac{3}{16} \\
& \geq \frac{3\log 2}{4096} \sqrt{\frac{k - 1}{n}}.
\end{align*}
\end{proof}

\begin{customthm}{\ref{thm:Wasserstein_distribution_estimation_lower_bound}}
Let $(\Omega,\rho)$ be a metric space, and let $\P$ denote the set of Borel probability measures on $(\Omega,\rho)$. Then,
\[\inf_{\hat P : \X^n \to \P} \sup_{P \in \P} \E_{X_1,...,X_n \IID P} \left[ W_r^r(P, \hat P(X_1,...,X_n)) \right]
  \geq c_r \sup_{k \in [32n]} R^r(k) \sqrt{\frac{k - 1}{n}},\]
where
\[c_r = \frac{3\log 2}{4096\cdot 2^r}.\]
is independent of $n$ and the infimum is taken over all estimators (i.e., all (potentially randomized) functions $\hat P : \X^n \to \P$ of the data).
\label{thm:Wasserstein_distribution_estimation_lower_bound_appendix}
\end{customthm}

\begin{proof}
Let $k \leq 32n$, and let $\D$ be an $R(k)$-packing $\D$ of $(\Omega,\rho)$ with $|\D| = k$. Let $\P_\D$ denote the class of (discrete) distributions over $\D$.
Applying Lemma~\ref{lemma:countable_support_bound}, Lemma~\ref{lemma:wasserstein_projections}, Lemma~\ref{lemma:multinomial_minimax_lower_bound}, and the definition of the packing radius (in that order)
\begin{align*}
\inf_{\hat P : \X^n \to \P} \sup_{P \in \P} \E \left[ W_r^r(\hat P, P) \right]
& \geq \left( \Sep(\D) \right)^r \inf_{\hat P : \X^n \to \P} \sup_{P \in \P} \E \left[ \|\hat P - P\|_1 \right] \\
& \geq \left( \Sep(\D) \right)^r \inf_{\hat P : \X^n \to \P} \sup_{P \in \P_\D} \E \left[ \|\hat P - P\|_1 \right] \\
& \geq \left( \frac{\Sep(\D)}{2} \right)^r \inf_{\hat P : \X^n \to \P_\D} \sup_{P \in \P_\D} \E \left[ \|\hat P - P\|_1 \right] \\
& \geq \frac{3\log 2}{4096 \cdot 2^r} \left( \Sep(\D) \right)^r \sqrt{\frac{|\D| - 1}{n}} \\
& \geq \frac{3\log 2}{4096\cdot 2^r} R^r(k) \sqrt{\frac{k - 1}{n}}.
\end{align*}
The theorem follows by taking the supremum over $k \leq 32n$ on both sides.
\end{proof}

\section{Proofs of Minimax Lower Bound in terms of Moment Bounds}
\label{sec:lower_bound_proof_moment_bounds}

In this section, we prove our second lower bound theorem (Theorem~\ref{thm:heavy_tailed_lower_bound}), for the case of distributions with unbounded support and bounded moments.

\begin{customthm}{\ref{thm:heavy_tailed_lower_bound}}
Suppose $r, \ell, \mu > 0$ are constants, and fix $x_0 \in \Omega$. Let $\P_{\ell,x_0}(\mu)$ denote the family of distributions $P$ on $\Omega$ with $\ell^{th}$ moment $\mu_{\ell,x_0}(P) \leq \mu$ around $x_0$ at most $\mu$. Let $n \geq \frac{3\mu}{2}$ and assume there exists $x_1 \in \Omega$ such that $\rho(x_0, x_1) = n^{1/\ell}$. Then,
\[M(r, \P_{\ell,x_0}(\mu)) \geq C_\mu n^{\frac{r - \ell}{\ell}},\]
where $C_\mu := \frac{\min \left\{ \mu, 2/3 \right\}}{24}$ is constant in $n$.
\end{customthm}

\begin{proof}
First, note a standard lemma for minimax lower bounds, which we reiterate in the case of Wasserstein distances:

\begin{lemma}[Theorem 2.1 of \citet{tsybakov2009introduction}, Wasserstein Case]
Assume there exist $P_0, P_1 \in \P$ such that $P_0 \ll P_1$ and $W_r^r(P_0, P_1) \geq 2s > 0$ such that $D_{KL} \left( P_0^n, P_1^n \right) \leq \frac{1}{2}$. Then,
\[\inf_{\hat P : \Omega \to \P} \sup_{P \in \P} \pr \left[ W_r^r \left( \hat P, P \right) \geq s \right]
  \geq \frac{1}{2} P_1 \left( \frac{dP_0}{dP_1}(x) \geq 1 \right).\]
  \label{lemma:tsybakov_lecam}
\end{lemma}
We now construct appropriate $P_0$ and $P_1$ to plug into the above lemma. Define
\[\epsilon := \frac{\min \left\{ \mu, 2/3 \right\}}{2n} \in (0,1/3],\]
and consider distinguishing between two discrete distributions
\[P_0 := \left( 1 - \epsilon \right) \delta_{x_0} + \epsilon \delta_{x_1}
	\quad \text{ and } \quad
	P_1 := \left( 1 - 2\epsilon \right) \delta_{x_0} + 2\epsilon \delta_{x_1}\]
where $\delta_x$ denotes a unit point mass at $x$.
Since, $\epsilon \in [0, 1/2]$, $P_0$ and $P_1$ are both probability distributions. Moreover, $\mu_{\ell,x_0} \left( P_0 \right) = \epsilon n \leq \mu/2$, and $\mu_{\ell,x_0} \left( P_1 \right) = 2\epsilon n \leq \mu$, so that $P_0, P_1 \in \P_{\ell,x_0}(\mu)$. Note that, since $\epsilon \leq 1/3$, by the inequality $\log(1 + x) \leq x$, we have
\[(1 - \epsilon) \log \frac{1 - \epsilon}{1 - 2 \epsilon}
  = (1 - \epsilon) \log \left( 1 + \frac{\epsilon}{1 - 2 \epsilon} \right)
  \leq (1 - \epsilon) \frac{\epsilon}{1 - 2 \epsilon}
  \leq 2 \epsilon.\]
Therefore,
\begin{align*}
D_{KL} \left( P_0^n, P_1^n \right)
  = n D_{KL} \left( P_0, P_1 \right)
& = n \left( P_0(x_0) \log \frac{P_0(x_0)}{P_1(x_0)} + P_0(x_1) \log \frac{P_0(x_1)}{P_1(x_1)} \right) \\
& = n \left( (1 - \epsilon) \log \frac{1 - \epsilon}{1 - 2 \epsilon} + \epsilon \log \frac{\epsilon}{2\epsilon} \right)
  \leq n \left( 2\epsilon - \epsilon \log 2 \right)
  \leq \frac{1}{2},
\end{align*}
since $2 - \log 2 \leq 3/2$. Finally, note that
\[W_r^r \left( P_0, P_1 \right)
  = \epsilon n^{r/\ell}
  = \min \left\{ \frac{\mu}{2}, \frac{1}{3} \right\} n^{\frac{r - \ell}{\ell}}.\]
Plugging $P_0$ and $P_1$ into Lemma~\ref{lemma:tsybakov_lecam} with $s = \min \left\{ \frac{\mu}{4}, \frac{1}{6} \right\} n^{\frac{r - \ell}{\ell}}$ thus gives
\[\inf_{\hat P : \Omega \to \P_{\ell,x_0}(\mu)} \sup_{P \in \P_{\ell,x_0}(\mu)} \pr \left[ W_r^r \left( \hat P, P \right) \geq s \right]
  \geq \frac{1}{2} P_1 \left( x_0 \right) = \frac{1 - 2\epsilon}{2} \geq 1/6.\]
Thus,
\[M(r, \P_{\ell,x_0}(\mu)) = \inf_{\hat P : \Omega \to \P_{\ell,x_0}(\mu)} \sup_{P \in \P_{\ell,x_0}(\mu)} \E_{X_1^n \IID P} \left[ W_r^r \left( \hat P, P \right) \right]
  \geq \frac{s}{6}
  = \frac{\min \left\{ \mu, 2/3 \right\}}{24} n^{\frac{r - \ell}{\ell}}.\]
\end{proof}

\section{Proofs of Technical Lemmas}
\label{app:proofs}

\begin{customlemma}{\ref{lemma:measures_agree_on_partition}}
Suppose $\S \in \SS$ is a countable Borel partition of $\Omega$. Let $P$ and $Q$ be Borel probability measures such that, for every $S \in \S$, $P(S) = Q(S)$. Then, for any $r \geq 1$, $W_r(P, Q) \leq \Res(\S)$.
\end{customlemma}

\begin{proof}
This fact is intuitively obvious; clearly, there exists a transportation map $\mu$ from $P$ to $Q$ that moves mass only within each $S \in \S$ and therefore without moving any mass further than $\delta$. For completeness, we give a formal construction.

Let $\mu : \Sigma^2 \to [0,1]$ denote the coupling that is conditionally independent given any set $S \in \S$ with $P(S) = Q(S) > 0$ (that is, for any $A, B \in \Sigma$, $\mu(A \times B \cap S \times S) P(S) = P(A \cap S) Q(B \cap S)$); the existence of such a measure can be verified by the Hahn-Kolmogorov theorem, similarly to that of the usual product measure (see, e.g., Section IV.4 of \citet{doob2012measure}). It is easy to verify that $\mu \in \mathcal{C}(P,Q)$. Since $\S$ is a countable partition and $\mu$ is only supported on $\bigcup_{S \in \S} S \times S$,
\begin{align*}
W_r(P, Q)
& \leq \left( \int_{\Omega \times \Omega} \rho^r(x,y) \, d\mu(x,y) \right)^{1/r} \\
& = \left( \sum_{S \in \S} \int_{S \times S} \rho^r(x,y) \, d\mu(x,y) \right)^{1/r} \\
& \leq \left( \sum_{S \in \S} \int_{S \times S} \delta^r \, d\mu(x,y) \right)^{1/r} \\
& = \delta \left( \sum_{S \in \S} \mu(S \times S) \right)^{1/r}
  = \delta \left( \sum_{S \in \S} \frac{P(S) Q(S)}{P(S)} \right)^{1/r}
  = \delta \left( \sum_{S \in \S} Q(S) \right)^{1/r}
  = \delta.
\end{align*}
\end{proof}

\begin{customlemma}{\ref{lemma:countable_support_bound}}
Suppose $(\Omega, \rho)$ is a metric space, and suppose $P$ and $Q$ are Borel probability distributions on $\Omega$ with countable support; i.e., there exists a countable set $\X \subseteq \Omega$ with $P(\X) = Q(\X) = 1$. Then, for any $r \geq 1$,
\[(\Sep(\X))^r \sum_{x \in \X} \left| P(\{x\}) - Q(\{x\}) \right|
  \leq W_r^r(P,Q)
  \leq (\Diam(\X))^r \sum_{x \in \X} \left| P(\{x\}) - Q(\{x\}) \right|.\]
\end{customlemma}

\begin{proof}
The term $\sum_{x \in \X} \left| P(\{x\}) - Q(\{x\}) \right| = TV(P, Q)$ is precisely the (unweighted) amount of mass that must be transported to transform between $P$ and $Q$. Hence, the result is intuitively fairly obvious; all mass moved has a cost of at least $\Sep(\Omega)$ and at most $\Diam(\Omega)$. However, for completeness, we give a more formal proof below.

To prove the lower bound, suppose $\mu \in \Pi(P, Q)$ is any coupling between $P$ and $Q$. For $x \in \X$,
\[\mu(\{x\} \times \{x\}) + \mu(\{x\} \times (\Omega \sminus \{x\}))
  = \mu(\{x\} \times \Omega)
  = P(\{x\})\]
and, similarly,
\[\mu(\{x\} \times \{x\}) + \mu((\Omega \sminus \{x\}) \times \{x\})
  = \mu(\Omega \times \{x\})
  = Q(\{x\}).\]
Since $P(\{x\}), Q(\{x\}) \in [0,1]$, it follows that
\[\mu(\{x\} \times (\Omega \sminus \{x\})) + \mu(\mu((\Omega \sminus \{x\}) \times \{x\}))
  \geq \left| P(\{x\} - Q(\{x\}) \right|.\]
Therefore, since $\rho(x,y) = 0$ whenever $x = y$ and $\rho(x, y) \geq \Sep(\Omega)$ whenever $x \neq y$,
\begin{align*}
\int_{\Omega \times \Omega} \rho^r(x, y) \, d\mu(x,y)
& = \int_{\X \times \X} \rho^r(x, y) \, d\mu(x,y) \\
& = \sum_{x \in \X} \int_{\{x\} \times (\Omega \sminus \{x\})} \rho^r(x, y) \, d\mu(x,y)
  + \int_{(\Omega \sminus \{x\}) \times \{x\}} \rho^r(x, y) \, d\mu(x,y) \\
& \geq (\Sep(\Omega))^r \sum_{x \in \X} \mu(\{x\} \times (\Omega \sminus \{x\}))
  + \mu((\Omega \sminus \{x\}) \times \{x\}) \\
& \geq (\Sep(\Omega))^r \sum_{x \in \X} \left| P(\{x\}) - Q(\{x\}) \right|.
\end{align*}
Taking the infimum over $\mu$ on both sides gives
\[(\Sep(\Omega))^r \sum_{x \in \X} \left| P(\{x\}) - Q(\{x\}) \right|
  \leq W_r^r(P, Q).\]
To prove the upper bound, since $\rho$ is upper bounded by $\Diam(\Omega)$, it suffices to construct a coupling $\mu$ that only moves mass into or out of each given point, but not both; that is, for each $x \in \X$,
\[\min\{\mu(\{x\} \times (\Omega \sminus \{x\})), \mu((\Omega \sminus \{x\}) \times \{x\})\}
  = 0.\]
One way of doing this is as follows. Fix an ordering $x_1,x_2,...$ of the elements of $\X$.
For each $i \in \N$, define
\[X_i := \sum_{\ell = 1}^i (P(x_\ell) - Q(x_\ell))_+
  \quad \text{ and } \quad
  Y_i := \sum_{\ell = 1}^i (Q(x_\ell) - P(x_\ell))_+,\]
and further define
\[j_i := \min \{ j \in \N : X_i \leq Y_j \}
  \quad \text{ and } \quad
  k_i := \min \{ k \in \N : X_j \geq Y_i \}.\]
Then, for each $i \in \N$, move $X_i$ mass from $\{x_1,...,x_i\}$ to $\{y_1,...,y_{j_i}\}$ and move $Y_i$ mass from $\{y_1,...,y_i\}$ to $\{x_1,...,x_{k_i}\}$. As $i \to \infty$, by construction of $X_i$ and $Y_i$, the total mass moved in this way is
\[\mu((\X \times \X) \sminus \{(x,x) : x \in \X\})
  = \lim_{i \to \infty} X_i + Y_i = \sum_{x \in \X} \left| P(x) - Q(x) \right|.\]
\end{proof}

\begin{customlemma}{\ref{lemma:nested_partitions_Wasserstein_bound}}
Let $K$ be a positive integer.
Suppose $\{\S_k\}_{k \in [K]}$ is a sequence of nested countable Borel partitions of $(\Omega,\rho)$, with $\S_0 = \Omega$. Then, for any $r \geq 1$ and any Borel probability distributions $P$ and $Q$ on $\Omega$,
\[W_r^r(P, Q)
  \leq (\Res(\S_K))^r + \sum_{k = 1}^K \left( \Res(\S_{k - 1}) \right)^r
                                      \left( \sum_{S \in \S_k} \left| P(S) - Q(S) \right| \right).\]
\end{customlemma}

\begin{proof}
Our proof follows the same ideas as and slightly generalizes of the proof of Proposition 1 in \citet{weed2017sharp}.
Intuitively, to prove Lemma~\ref{lemma:nested_partitions_Wasserstein_bound} it suffices to find a transportation map such that
For each $k \in [K]$, recursively define
\[P_k := P - \sum_{j = 0}^{k - 1} \mu_k
  \quad \text{ and } \quad
  Q_k := Q - \sum_{j = 0}^{k - 1} \nu_k,\]
where, for each $k \in [K]$, $\mu_k$ and $\nu_k$ are Borel measures on $\Omega$ defined for any $E \in \Sigma$ by
\[\mu_k(E) := \sum_{S \in \S_k : P_k(S) > 0} \left( P_k(S) - Q_k(S) \right)_+ \frac{P_k(E \cap S)}{P_k(S)}\]
and
\[\nu_k(E) := \sum_{S \in \S_k : Q_k(S) > 0} \left( Q_k(S) - P_k(S) \right)_+ \frac{Q_k(E \cap S)}{Q_k(S)}.\]

By construction of $\mu_k$ and $\nu_k$, each $\mu_k$ and $\nu_k$ is a non-negative measure and $\sum_{k = 1}^K \mu_k \leq P$ and $\sum_{k = 1}^K \nu_k \leq Q$. Furthermore, for each $k \in [K - 1]$, for each $S \in \S_k$, $\mu_{k + 1}(S) = \nu_{k + 1}(S)$, and
\[\mu_k(\Omega) = \nu_k(\Omega) \leq \sum_{S \in \S_k} \left| P(S) - Q(S) \right|.\]
Consequently, although $\mu$ and $\nu$ are not probability measures, we can slightly generalize the definition of Wasserstein distance by writing
\[W_r^r \left( \mu_k, \nu_k \right)
  := \mu(\Omega) \inf_{\tau \in \Pi \left( \frac{\mu_k}{\mu_k(\Omega)}, \frac{\nu_k}{\nu_k(\Omega)}\right)} \E_{(X,Y) \sim \mu} \left[ \rho^r \left( X, Y \right) \right]\]
(or $W_r^r(\mu_k, \nu_k) = 0$ if $\mu_k = \nu_k = 0$).
In particular, this is convenient because we one can easily show that, by construction of the sequences $\{P_k\}_{k \in [K]}$ and $\{Q_k\}_{k \in [K]}$,
\begin{equation}
W_r^r(P, Q)
  \leq W_r^r \left( P_K, Q_K \right) + \sum_{k = 1}^K W_r^r \left(\mu_k, \nu_k \right).
  \label{ineq:decomposition}
\end{equation}
For each $k \in [K]$, Lemma~\ref{lemma:countable_support_bound} implies that
\begin{align*}
W_r^r(\mu_k,\nu_k)
& \leq \sum_{S \in \S_{k - 1}} \left( \Diam(S) \right)^r \sum_{T \in \S_k : T \subseteq S} \left| P(T) - Q(T) \right| \\
& \leq \left( \Res(\S_{k - 1}) \right)^r \sum_{S \in \S_{k - 1}} \sum_{T \in \S_k : T \subseteq S} \left| P(T) - Q(T) \right| \\
& = \left( \Res(\S_{k - 1}) \right)^r \sum_{T \in \S_k} \left| P(T) - Q(T) \right|.
\end{align*}
Furthermore, for each $S \in \S_K$, $P_K = Q_K$, Lemma~\ref{lemma:measures_agree_on_partition} gives that
\[W_r^r \left( P_K, Q_K \right)
  \leq \left( \Res(\S_K) \right)^r\]
Plugging these last two inequalities into Inequality~\eqref{ineq:decomposition} gives the desired result:
\[W_r^r(P, Q)
  \leq \left( \Res(\S_K) \right)^r + \sum_{k = 1}^K \left( \Res(\S_{k - 1}) \right)^r \sum_{S \in \S_k} \left| P(S) - Q(S) \right|.\]
\end{proof}

\begin{customlemma}{\ref{lemma:fine_refinement}}
Suppose $\S$ and $\T$ are partitions of $(\Omega,\rho)$, and suppose $\S$ is countable. Then, there exists a partition $\S'$ of $(\Omega,\rho)$ such that:
\begin{enumerate}[label=\alph*)]
\item
$|\S'| \leq |\S|$.
\item
$\Res(\S') \leq \Res(\S) + 2\Res(\T)$.
\item
$\T$ is a refinement of $\S'$.
\end{enumerate}
\end{customlemma}
\begin{proof}
Enumerate the elements of $\S$ as $S_1,S_2,...$. Define $S_0' := \emptyset$, and then, for each $i \in \{1,2,...\}$, recursively define
\[S_i' := \left. \left( \bigcup_{T \in \T : T \cap S_i \neq \emptyset} T \right) \middle \sminus \left( \bigcup_{j = 1}^{i - 1} S_j' \right) \right.,\]
and set $\S' = \{S_1',S_2',...\}$. Clearly, $|\S'| \leq |\S|$ (equality need not hold, as we may have some $S_i' = \emptyset$).
By the triangle inequality, each
\[\Diam(S_i')
  \leq \Diam \left( \bigcup_{T \in \T : T \cap S_i \neq \emptyset} T \right)
  \leq \delta_\S + 2 \delta_T.\]
Finally, since $\T$ is a partition and we can write
\[S_i' = \left. \left( \bigcup_{T \in \T : T \cap S_i \neq \emptyset} T \right) \middle \sminus \left( \bigcup_{j = 1}^{i - 1} \bigcup_{T \in \T : T \cap S_j' \neq \emptyset} T \right) \right.,\]
$\T$ is a refinement of $\S'$.
\end{proof}

\end{document}